\documentclass[final,3p,times,12pt]{elsarticle}

\usepackage{lipsum}
\usepackage{amsfonts}
\usepackage{graphicx}
\usepackage{epstopdf}
\usepackage{algorithmic}
\ifpdf
  \DeclareGraphicsExtensions{.eps,.pdf,.png,.jpg}
\else
  \DeclareGraphicsExtensions{.eps}
\fi

\usepackage{amsmath,amsfonts,amssymb,amscd,amsbsy}
\usepackage{stmaryrd}
\usepackage{lastpage}
\usepackage{enumerate}
\usepackage{enumitem}
\usepackage{fancyhdr}
\usepackage{mathrsfs}
\usepackage{xcolor}
\usepackage{graphicx}
\usepackage{hyperref}
\usepackage{subfig}
\usepackage{float}
\usepackage{enumitem}

\newcommand{\R}{\mathbb{R}}
\newcommand{\C}{\mathbb{C}}

\renewcommand{\vec}[1]{\mathbf{#1}}

\newcommand{\Sw}{\mathcal{S}}
\newcommand{\Cminus}{\C_{-}}

\usepackage{amsopn}

\usepackage{amssymb}
\usepackage{amsmath}
\usepackage{amsthm}
\usepackage{comment}
\usepackage{soul}

\newtheorem{lemma}{Lemma}

\newtheorem{theorem}{Theorem}
\newtheorem{corollary}{Corollary}

\usepackage{lineno}

\begin{document}

\begin{frontmatter}



\title{Convergence of the Semi-Discrete WaveHoltz Iteration}


\affiliation[kth]{organization={Department of Mathematics, Royal Institute of Technology},
            city={Stockholm},
            country={Sweden}}

\affiliation[vt]{organization={Department of Mathematics, Virginia Tech},
            city={Blacksburg},
            state={VA},
            country={United States}}
            
\author[vt]{Amit Rotem}
\author[kth]{Olof Runborg}
\author[vt]{Daniel Appel\"{o}}

\begin{abstract} 
Solving the Helmholtz equation with iterative methods is challenging because of its indefinite nature and highly oscillatory solutions. 
The WaveHoltz algorithm mitigates the difficulties of solving the Helmholtz equation by repeatedly solving the wave equation over short periods in time.
In this paper, we prove that for stable semi-discretizations of the wave equation, the WaveHoltz iteration converges to an approximate solution of the corresponding frequency-domain problem, provided one exists. We present numerical examples in one and two dimensions using finite difference and discontinuous Galerkin discretizations illustrating these convergence results.
\end{abstract}




\end{frontmatter}


\section{Introduction}
The wave equation is a fundamental mathematical model for describing the propagation of information in physical phenomena arising in acoustics, electromagnetism, optics, and fluid dynamics. Its numerical simulations are extensively used to advance research and develop new tools in, for example,  communication, imaging, seismic modeling, and engineering.

When simulating the wave equation in the frequency domain, we encounter the Helmholtz equation,
\begin{equation}
   \nabla \cdot (c^2(x) \nabla \hat{u}) + \omega^2 \hat{u} = f(x),\label{eq:HELMHOLTZ} 
\end{equation}
where $c$ is the wave speed and $f$ is a source function.
In this paper we consider \eqref{eq:HELMHOLTZ} 
in $d$ dimensions, posed on a
smooth bounded domain $\Omega\subset{\mathbb R}^d$,
complemented with suitable boundary conditions on $\partial\Omega$. A typical algorithm for approximating solutions to \eqref{eq:HELMHOLTZ} discretizes it using finite differences or finite elements, resulting in a linear system of equations
\[ (L - \omega^2 M) \hat{u}_h = -f_h, \]
for the representation of the approximate solution $\hat{u}_h$. Here $L$ approximates the negative Laplacian ($L \hat{u}_h\approx -\nabla\cdot(c^2 \nabla \hat{u})$) and $M$ is either a mass matrix in the finite element context or the identity in the finite difference context. To this end, the subscript $h$ identifies the quantities that are approximations.  For energy conserving boundary conditions (Dirichlet or Neumann), the matrix $L$ is often symmetric and positive semi-definite. Similarly $M$ is necessarily symmetric positive definite, therefore the matrix $L - \omega^2 M$ is usually indefinite. For boundary conditions which dissipate energy, $L$ may be complex symmetric (not Hermitian) or not symmetric at all.

When $\omega$ is large, the solution is highly oscillatory and it needs to be resolved by many degrees of freedom. To sufficiently resolve the solution by a finite difference or finite element method of order $p$ in $d$ dimensions we require
\begin{equation}\label{eq:homegascale}
     \omega^{p+1}h^p = \mathrm{const.} \implies \#\mathrm{DOFs} \sim O(\omega^{d(p+1)/p}).
\end{equation}
Here $h$ is the scale of the mesh \cite{WDHthumb,Advancements-Iterative-Helmholtz-Erlangga}. For high frequency problems, the number of degrees of freedom can easily exceed tens of thousands in two dimensions, and millions in three dimensions. 

The indefinite nature of the Helmholtz equation makes it challenging for classical iterative algorithms.
Highly efficient iterative methods for the similar (yet definite) Poisson equation,
\[ -\nabla\cdot (c^2(x) \nabla \hat{u}) = f, \]
often fail for the Helmholtz equation. For example, since the Helmholtz equation is indefinite the popular conjugate gradient method is inapplicable so other methods such as GMRES or the biconjugate method must be used instead. Yet still, these methods converge much more slowly for the indefinite problem than for the definite one. Other iterative methods, developed for Poisson's equation, such as Schwartz domain decomposition or multigrid, entirely fail to converge for the Helmholtz equation and fail to substantially improve convergence as preconditioners for GMRES \cite{Helmholtz-Difficult-Ernst-Gander, Advancements-Iterative-Helmholtz-Erlangga}.

To remedy this difficulty, specialized algorithms for the Helmholtz equation have been introduced. For example, the domain decomposition method of Despr\'{e}s which replaces the Dirichlet (or Neumann) condition of the classical Schwartz iteration with impedance conditions \cite{Domain-Decomp-Depres}, or the wave-ray multigrid method of Brandt and Livshits in which the typical smoothing step is replaced by ray-cycles \cite{Wave-Ray-Multigrid-Brandt-Livshits} (see also \cite{Helmholtz-Difficult-Ernst-Gander, Advancements-Iterative-Helmholtz-Erlangga}). Additionally, there have been extensive developments of preconditioning techniques such as ILU factorizations, or methods which solve the complex-shifted Laplacian as a means to control the distribution of the eigenvalues of the preconditioned system. The sweeping preconditioner of Engquist and Ying (see \cite{Sweeping-Precond-Engquist-Ying}) is another example of highly successful specialized algorithm. 

There are also several direct methods for solving the Helmholtz equation, for example the Hierarchically Semi-Separable (HSS) parallel multifrontal sparse solver by deHoop and co-authors, \cite{2011_deHoop_Xia}, and the spectral collocation solver by Gillman, Barnett and Martinsson, \cite{Gillman2015}. Of course, when the material properties are piecewise constant integral equations can be used. These have a long history and a complete literature review is beyond the scope of this introduction and we refer the reader to the review \cite{BEMReview} and its 400+ references. 

An entirely different set of algorithms solve the Helmholtz equation by solving a closely related wave equation in the time domain. Controllability methods, for example, solve the Helmholtz equation by minimizing the difference between the initial condition and the solution of the wave equation after one period ($2\pi/\omega$) in time. The solution to the Helmholtz equation is the initial condition that minimizes this difference, and the minimization problem can be solved with the conjugate gradient method. The original controllability method can be found in  \cite{bristeau1998controllability} and more recent developments and extensions can be found in \cite{grote2019controllability,GlowRoss06,GROTE2020112846}. 

Another iterative method, and the focus of this paper, is the WaveHoltz algorithm which solves the Helmholtz equation by repeatedly filtering the solution to the wave equation in the time domain. The resulting fixed point iteration can be accelerated by Krylov methods, and, in contrast to classical discretizations, the system of equations associated with the fixed point of WaveHoltz is positive definite for energy conserving boundary conditions. When energy is not necessarily conserved, the WaveHoltz iteration still benefits from the efficiency of high performance wave equation solvers.

In the first part of this paper we present a new result which guarantees convergence of the WaveHoltz iteration itself, for any stable semi-discretization of the wave equation, to the approximate solution of the Helmholtz equation, provided it exists and is unique. More precisely, this result requires that the eigenvalues of the discretization matrix have non-positive real parts and are not equal to $\pm i\omega$, with the convergence rate improving as the smallest distance to $\pm i\omega$ increases.
After proving these theoretical results, we will consider a few numerical examples in which these results are realized. First we will investigate one dimensional problems with impedance boundary conditions.
We will investigate the distribution of eigenvalues of the discretization matrix and find that, indeed, the eigenvalues of these matrices are bounded away from $\pm i \omega$.
Afterwards, we will investigate the convergence rates for similar problems in two dimensions.
For results relating to closed domains (energy conserving boundary conditions) we refer the reader to the original paper, \cite{WaveHoltz} and for other extensions of the WaveHoltz approach we refer to \cite{multiWHI,EigenWave,appelo2025optimalonhelmholtzsolver,peng2021emwaveholtz,ElWaveHoltz}.

The rest of the paper is organized as follows. In Section \ref{sec:helmholtz} we will introduce the frequency domain problem and relate it to the time domain problem. We will then define the semi-discrete WaveHoltz iteration and show that the solution to the discretized Helmholtz equation is a fixed point of the iteration. In Section \ref{sec:conv-theory} we prove the convergence result for the semi-discrete WaveHoltz iteration.
Lastly, in Section \ref{sec:numeric} we perform numerical experiments which relate the theoretical results to two common discretizations of the wave equation: a finite difference and a discontinuous Galerkin method.

\section{The Discrete Helmholtz Equation} \label{sec:helmholtz}

In this paper, our goal is to solve discrete Helmholtz equations of the form:
\begin{equation} \label{eq: DISCRETE HELMHOLTZ}
    i\omega \hat{w}_h = A \hat{w}_h - F.
\end{equation}
Here $\hat{w}_h, F\in\C^m$, and $A\in \C^{m\times m}$.
Such equations appear from the discretizations of the Helmholtz equation:
\begin{equation} \label{eq: HELMHOLTZ}
    \nabla \cdot (c^2(x) \nabla \hat{u}) + \omega^2 \hat{u} = f(x).
\end{equation}
Consider the function $w(t) = \hat{w}_h e^{i\omega t}$. If we differentiate $w$ in time we find
\[ w_t = i\omega \hat{w}_h e^{i\omega t} = (A \hat{w} - F)e^{i\omega t}. \]
That is, $w$ satisfies
\begin{equation} \label{eq: SEMI-DISCETE WAVE}
    w_t = Aw - F e^{i\omega t}.
\end{equation}
We refer to \eqref{eq: SEMI-DISCETE WAVE} as the semi-discrete wave equation because it may also be derived from spatial semi-discretization of the wave equation
\begin{equation} \label{eq:WAVE_EQ}
    u_{tt} = \nabla\cdot(c^2(x) \nabla u) - f(x)e^{i\omega t}.
\end{equation}

The wave equation is often easier to solve (for short time ranges) because explicit time-stepping schemes are memory-efficient, implicit methods can be usually accelerated with multigrid, and alternative approaches like parallel-in-time methods have also proven effective \cite{Gander2019-jy, Newmark1962}. The aim of the WaveHoltz algorithm is to 
leverage the efficiency of wave equation solvers to solve the Helmholtz equation.

In Section \ref{sec: semi discrete iteration}, we will introduce the WaveHoltz algorithm as it is applied in a ``semi-discrete'' setting, where we have discretized in space but remain continuous in time. In Section \ref{sec: discrete to continuous} we connect the semi-discrete problem to the target application, the Helmholtz equation by writing the wave equation as a \emph{first order} system in time. Such a formulation is very general because the wave equation naturally appears as a first order system in the context of hyperbolic conservation laws, while second order equations are easily converted to first order by introducing a ``velocity'' variable.

\subsection{The Semi-Discrete WaveHoltz Iteration} \label{sec: semi discrete iteration}
Consider the discrete Helmholtz equation \eqref{eq: DISCRETE HELMHOLTZ}. The WaveHoltz iterative algorithm \cite{WaveHoltz} approximates the solution to the Helmholtz equation by filtering solutions of the time dependent wave equation in the fixed point iteration,
\begin{equation} \label{eq:semi-discrete-iteration}
    \hat{w}_h^{(n+1)} = \Pi_h \hat{w}_h^{(n)},
\end{equation}
where $\Pi_h$ is defined as
\begin{equation} \label{eq:semi-pi}
\Pi \hat{w}_h = \frac{2}{T}\int_0^T \left(\cos(\omega t) - \frac{1}{4}\right) w_h(t) \, dt, \qquad T = \frac{2\pi}{\omega},
\end{equation}
and $w_h(t)$ solves the semi-discrete wave equation \eqref{eq: SEMI-DISCETE WAVE} with initial conditions
\begin{gather*}
    w_h(0) = \hat{w}_h.
\end{gather*}
The integration kernel is designed to damp all time harmonic modes of the solution with frequency not equal to $\omega$, while preserving the amplitude of the $\omega$ frequency mode.
To see the latter, let $\hat{w}_h^\star$ be the solution to \eqref{eq: DISCRETE HELMHOLTZ} and set the initial condition to the wave equation as $w_h(0) = \hat{w}_h^\star$, then the solution is $w_h(t) = \hat{w}^\star_h e^{i\omega t}$. Because
\begin{equation} \label{eq:filterunit}
\frac{2}{T}\int_0^T \left(\cos(\omega t) - \frac{1}{4}\right) 
e^{i\omega t}\, dt    =1,
\end{equation}
the integral in \eqref{eq:semi-pi} evaluates to $\hat{w}_h^\star$.
Thus, by construction, the solution $\hat{w}_h^\star$ to the discrete Helmholtz equation corresponds to the fixed-point
\[ \hat{w}_h^\star = \Pi \hat{w}_h^\star. \]
In addition, the operator $\Pi_h$ is affine, so $\Pi_h \hat{w}_h$ can be decomposed as the action of a linear operator $\Sw_h$ plus a fixed vector $\pi_{0,h}$ which depends only on $F$, that is,
\begin{equation} \label{eq:pi-affine}
    \Pi_h \hat{w}_h = \Sw_h \hat{w}_h + \pi_{0,h}, 
\qquad \pi_{0,h} = \Pi_h \vec{0},
\end{equation}
where
\begin{subequations} \label{eq: Sh}
    \begin{gather}
        \Sw_h \hat{w}_h = \frac{2}{T}\int_0^T \left( \cos(\omega t) - \frac{1}{4} \right) w_h(t) \, dt, \\
        \frac{dw_h}{dt} = A w_h, \ \ \ \ w_h(0) = \hat{w}_h. \label{eq:hom ODE}
    \end{gather}
\end{subequations}
Note that wave equation \eqref{eq:hom ODE} here is homogeneous.

\subsection{Connection to the Continuous Problem} \label{sec: discrete to continuous}

To relate the semi-discrete iteration to the solution of \eqref{eq:HELMHOLTZ} and \eqref{eq:WAVE_EQ}, we begin by rewriting equation \eqref{eq:WAVE_EQ} as a first order system in time. We consider two possible representations. Either, we introduce the velocity variable $v$ and write
\begin{subequations} \label{eq:system_A}
    \begin{align}
        u_t &= v, \\
        v_t &= \nabla \cdot (c^2 \nabla u) - f(x) e^{i\omega t}.
    \end{align}
\end{subequations}
Alternatively, we write the wave equation as the linear first order hyperbolic system:
\begin{subequations} \label{eq:system_B}
    \begin{align}
        p_t &+ \nabla \cdot (c^2\mathbf{u}) = -\frac{1}{i\omega}f(x)e^{i\omega t}, \label{eq:system_B pressure} \\
        \mathbf{u}_t &+ \nabla p = 0. \label{eq:system_B velocity}
    \end{align}
\end{subequations}
Formally, observe that differentiating \eqref{eq:system_B pressure} in time, switching the order of differentiation, and substituting \eqref{eq:system_B velocity} we recover
\[ p_{tt} = \nabla\cdot(c^2\nabla p) - f(x)e^{i\omega t}, \]
which is our original wave equation in the variable $p$. Similar arguments can be made for other first order systems (e.g. Maxwell's equations is a notable application, see \cite{peng2021emwaveholtz}), but we will focus on these two systems as representative examples.

In Section \ref{sec:FDM}, we will approximate the system, \eqref{eq:system_A}, with a finite difference method, and in Section \ref{sec:DGM} we will approximate the system, \eqref{eq:system_B}, using a Discontinuous Galerkin method. In both cases, once the equations have been discretized in space they can be written as a first order system of ordinary differential equations.

For example, let $w_h \in \C^{m}$ be the spatial degrees of freedom which collectively represents approximations to both $u$ and $v$ in system \eqref{eq:system_A} or $p$ and $\mathbf{u}$ in system \eqref{eq:system_B}. The semi-discretization of the wave equation (on either form) can then be written
\begin{equation} \label{eq:SEMIDISCRETE_WAVE_EQ}
   \frac{dw_h}{dt} = A w_h - Fe^{i\omega t}.
\end{equation}
Here $A\in\C^{m\times m}$ is an approximation of the spatial 
operators 
\begin{equation}\label{eq:firstorderops}
    \begin{pmatrix}
0 & I \\
  \nabla\cdot c^2 \nabla & 0\\
\end{pmatrix} \qquad\text{or}\qquad
\begin{pmatrix}
0 & -\nabla\cdot c^2 \\
  -\nabla & 0\\
\end{pmatrix},
\end{equation}
corresponding to \eqref{eq:system_A} and \eqref{eq:system_B},
and $F\in\C^m$ contains an approximation of 
$$
\begin{pmatrix}
0\\
f
\end{pmatrix} \qquad\text{or}\qquad
\begin{pmatrix}
\frac{1}{i\omega}f \\
  0\\
\end{pmatrix}.
$$
More details are given below in Section \ref{sec:numeric} and in \ref{apx:real-valued-wh}.

\subsection{Aside: WaveHoltz, Helmholtz and Their Relation} \label{sec:aside}
As an aside, we note that the applying the WaveHoltz iteration may be thought of as applying a particular preconditioner to the Helmholtz problem \eqref{eq: DISCRETE HELMHOLTZ}. To see this, note that the solution to \eqref{eq: SEMI-DISCETE WAVE} with initial condition $w_0$ has the exact solution:
\[
w_h(t) = e^{tA} w_0 + (e^{i\omega t} I - e^{tA}) (A - i\omega I)^{-1} F = e^{t A} w_0 + (e^{i\omega t} I - e^{tA}) \hat{w}_h^\star.
\]
Hence, the solution explicitly encodes the solution of the Helmholtz equation \eqref{eq: DISCRETE HELMHOLTZ}. Applying the filter \eqref{eq:semi-pi}, we find that:
\[
\Pi_h w_0 = \mathcal{S}_h w_0 + (I - \mathcal{S}_h) \hat{w}^\star_h = \mathcal{S}_h w_0 + (I - \mathcal{S}_h) (A - i\omega I)^{-1} F.
\]
Here $\mathcal{S}_h$ is previously defined in \eqref{eq: Sh}.
The linear system for the fixed point is therefore
\[
(I - \mathcal{S}_h) \hat{w}_h^\star = (I - \mathcal{S}_h) (A - i\omega I)^{-1} F.
\]
Defining the matrix $\mathcal{G}_h = (I - \mathcal{S}_h) (A - i\omega I)^{-1}$,
it follows that
\[
\mathcal{G}_h  (A - i\omega I) \hat{w}_h^\star = \mathcal{G}_h F.
\]
Thus, using the WaveHoltz method is equivalent to applying a left preconditioner $\mathcal{G}_h$ to the Helmholtz equation
\eqref{eq: DISCRETE HELMHOLTZ}.

\section{Convergence Theory for the Semi-Discrete Iteration} \label{sec:conv-theory}

In this section, we present convergence results for the semi-discrete WaveHoltz iteration \eqref{eq:semi-discrete-iteration}. 
To obtain convergence
we first need to
assume that the discretization
of the semi-discrete problem
\eqref{eq: SEMI-DISCETE WAVE} is stable and that
the discrete Helmholtz equation \eqref{eq: DISCRETE HELMHOLTZ} has a unique solution.
These two conditions
boil down to two
assumptions about the eigenvalues
of the discretization matrix $A$
that appears in both
\eqref{eq: SEMI-DISCETE WAVE}
and
\eqref{eq: DISCRETE HELMHOLTZ},
namely:
\begin{enumerate}[label=(\textbf{A\arabic*})]
    \item\label{assumption:stability} (Stability.)
    The eigenvalues $\{\lambda_j\}_{j=1}^m$ of $A$ have a non-positive real part.
    \item\label{assumption:E+U} (Existence and
    uniqueness.)
    No eigenvalue of $A$
    equals $i\omega$ or $-i\omega$.
\end{enumerate}
To analyze the convergence, 
we first recall that \eqref{eq:semi-discrete-iteration} 
can be written as the linear fixed point iteration, 
$$
\hat{w}_h^{n+1}=\Sw_h \hat{w}^n_h + \pi_{0,h}, 
\qquad \pi_{0,h} = \Pi_h \vec{0},
$$
by (\ref{eq:pi-affine})
with 
$\Sw_h$ defined in \eqref{eq: Sh}. 
For the iteration to converge, 
all eigenvalues of $\Sw_h$ therefore needs to 
lie inside the unit disc in the complex plane.

The main insight is that the eigenvalues $\{\mu_j\}$ of $\Sw_h$ are directly
related to the eigenvalues $\{\lambda_j\}$ of the matrix $A$ via 
a simple function. More precisely, we show below that
\begin{equation} \label{eq: beta def}
\mu_j = \hat\beta\left(\lambda_j / \omega\right), \quad
 \hat{\beta}(z) = \frac{1}{\pi} \int_0^{2\pi} \left( \cos(s) - \frac{1}{4} \right) e^{z s} \, ds. 
\end{equation}
We refer to this function as the {\it scaled filter transfer function}.
It is plotted in Figure~\ref{fig:BETA}.
\begin{figure}[tbp]
    \centering
    \includegraphics[width=0.3\textwidth,trim={0.0cm 0.0cm 0.0cm 0.0cm},clip]{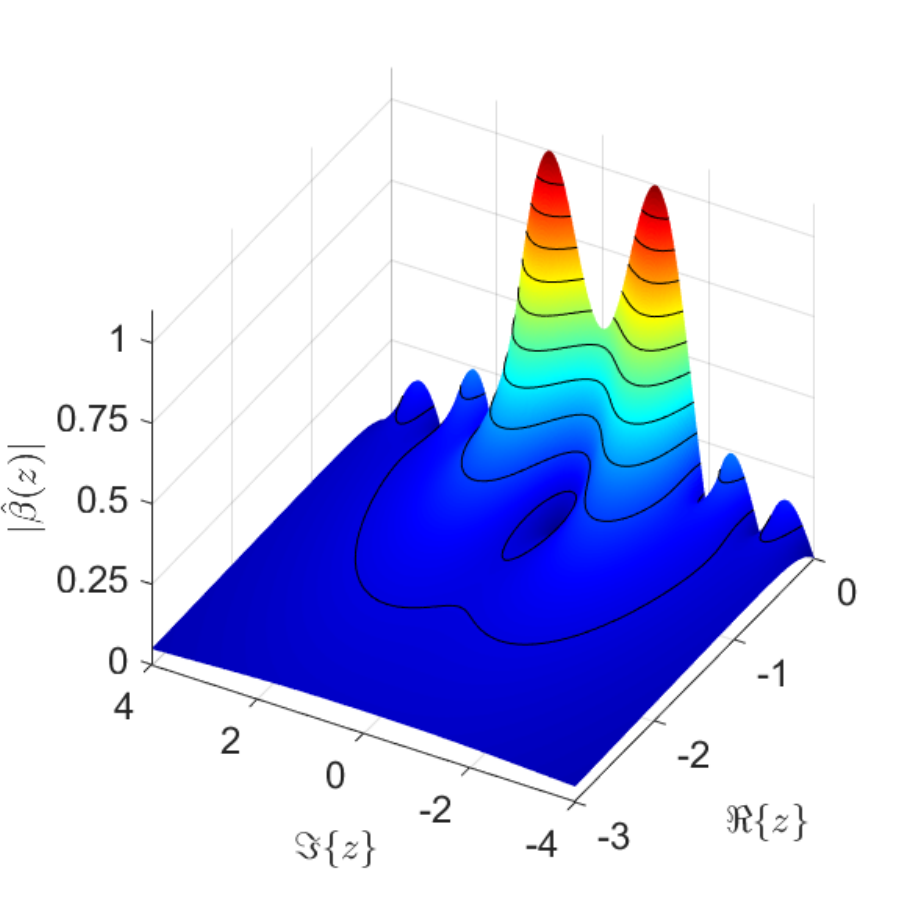}
    \includegraphics[width=0.3\textwidth,trim={0.0cm 0.0cm 0.0cm 0.0cm},clip]{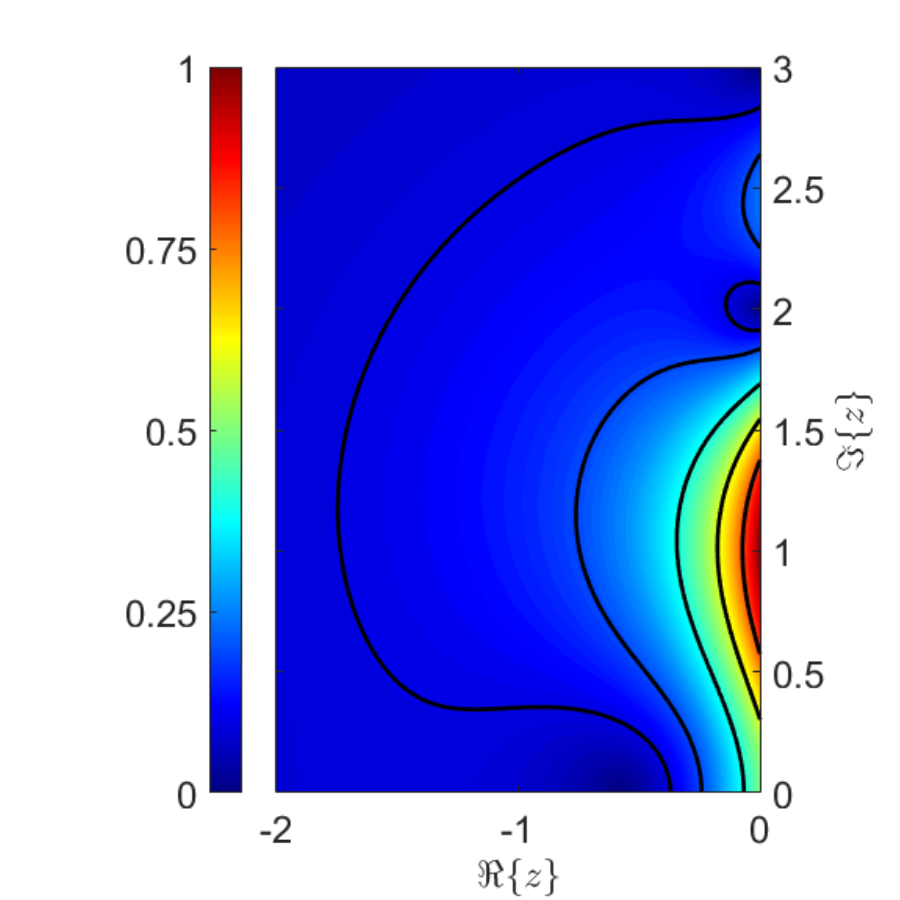}
    \includegraphics[width=0.3\textwidth,trim={0.0cm 0.0cm 0.0cm 0.0cm},clip]{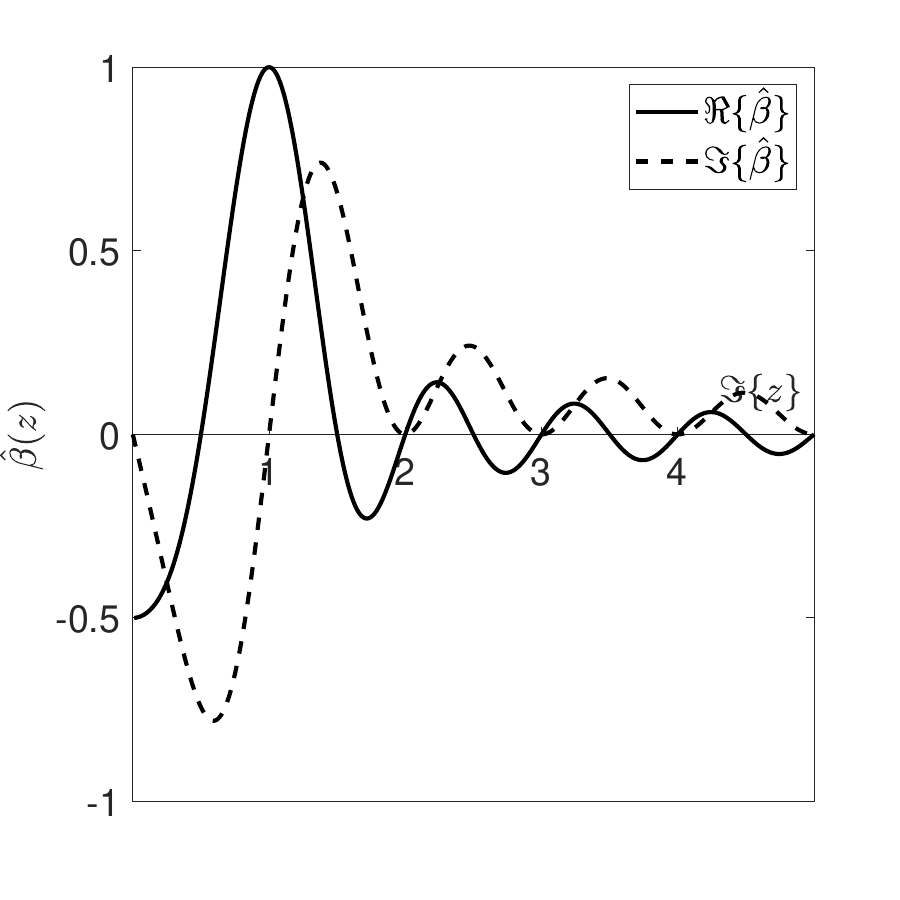}
    \caption{The three figures show the modulus of $\hat{\beta}$ in the left half-plane (left), its level contours near $z=i$ which are approximately parabolic (middle), and its restriction to the imaginary axis (right)
    \label{fig:BETA}}
\end{figure}

In Lemma \ref{lemma: beta complex} we show that 
$\hat\beta$ maps the left half-plane to the unit
disc in the complex plane. Additionally, we show that
$|\hat\beta(\lambda/\omega)|=1$ only when $\lambda=\pm i\omega$.
The implication of this is that the spectral radius
of $\Sw_h$ is strictly less than one
if all eigenvalues of $A$
lie in the left halfplane, away from $\pm i\omega$.
Or, put another way, the iteration converges if \ref{assumption:stability} and \ref{assumption:E+U} hold.

The {\it rate} of convergence is given by the 
magnitude of the largest eigenvalue
of $\Sw_h$. This depends on how close the eigenvalues of $A$
are to $\pm i\omega$ when measured with the following parabolic
distance function
\begin{equation} \label{eq: distance def}
d(z_1, z_2) := |\Re\{z_1 - z_2\}| + \alpha |\Im\{z_1 - z_2\}|^2, 
\qquad \alpha = \frac{2\pi^2-3}{12\pi}\approx 0.44.
\end{equation}
Indeed, in the proof of Lemma \ref{lemma: beta complex} we show that
$$
   |\mu_j|^2 = \left|\hat\beta\left(\lambda_j/\omega\right)\right|^2 \approx 
   1-2\pi \cdot d\left(\lambda_j/\omega, i\right).
$$
The significance of the parabolic distance is also visible in
in the middle subplot of Figure~\ref{fig:BETA}
which displays level contours of $|\hat{\beta}|$. It can be seen that the decrease in $|\hat{\beta}(z)|$ depends on the direction in which $z$ move away from $\pm i$.
Specifically, the figure shows that near $z = \pm i$, $|\hat\beta|$ is more sensitive to changes in the direction of the real axis than in the direction of the imaginary axis.
The shape of the level contours thus motivates the introduction of a parabolic distance function. The constant $\alpha$ in \eqref{eq: distance def}
is chosen so that the level contours of $|\hat{\beta}|$ align with those of $d(z,\pm i)$ near $z = \pm i$, this will be made clear in the proof of Lemma \ref{lemma: beta complex}.

We are now ready to present 
the main result establishing convergence in the following theorem.

\begin{theorem}\label{maintheorem}
Assume that \ref{assumption:stability} and \ref{assumption:E+U} hold
and additionally that
$A \in \C^{m\times m}$ can be diagonalized as $A = R \Lambda R^{-1}$.
Let $\hat{w}^\star_h \in \C^m$ be the solution of \eqref{eq: DISCRETE HELMHOLTZ} and let $\hat{w}_h^{(n)}\in \C^m$ be the $n^{th}$ iterate generated by the WaveHoltz iteration \eqref{eq:semi-discrete-iteration} with initial condition $\hat{w}_h^{(0)}$.
Define
$$
  \varepsilon = \min_j \min\left\{
  d(\lambda_j/\omega,i),
  d(\lambda_j/\omega,-i)
  \right\} > 0.
$$
Then, $\hat{w}^{(n)}_h\to \hat{w}_h^\star$ with a convergence rate $r\in(0,1)$ satisfying
\begin{equation}\label{eq:rassumption}
  r \leq \max\left\{1 - \varepsilon, 1 - \delta\right\},
\end{equation}
and $\delta \in (0,1)$ a universal constant independent of $A$, $F$, and $\omega$. 
Furthermore, let $||\cdot||$ be any $p$-norm.
Then at each iteration $n$,
\begin{equation}\label{eq:oneiterconv} 
\|R^{-1}(\hat{w}^{(n+1)}_h - \hat{w}_h^\star)\| \leq 
r\|R^{-1}(\hat{w}^{(n)}_h - \hat{w}_h^\star)\|,
\end{equation}
and, if $\kappa(R) = \|R\|\cdot\|R^{-1}\|$ is the condition number of $R$, 
\begin{equation} \label{eq:manyiter}
    \|\hat{w}_h^{(n)} - \hat{w}_h^\star\| \leq \kappa(R) r^n \|\hat{w}_h^{(0)} - \hat{w}_h^\star\|.
\end{equation}
\end{theorem}

We now make some observations about Theorem \ref{maintheorem} and its consequences.
We start with some comments on the convergence.
\begin{itemize}
\item Under the assumptions \ref{assumption:stability} and \ref{assumption:E+U} the theorem guarantees that
the WaveHoltz iteration converges as \mbox{$n\to \infty$}.
Note that a discretization of the Helmholtz equation is non-singular
if $\lambda_j\neq \pm i\omega$ for all $j$,
i.e under assumption \ref{assumption:E+U}.
 Hence, as the number
of eigenvalues is finite, $\varepsilon > 0$ for
any non-singular discretization of the Helmholtz equation.
\item 
The parameter $\varepsilon$ measures the relative gap
between the eigenvalues of the system and $i\omega$
in terms of the parabolic distance $d$.
The relative gap is a consequence of the discretization and the physical problem setup rather than something introduced by the WaveHoltz iteration.

\item
The precise estimate  (\ref{eq:oneiterconv}) is for the  coefficients $R^{-1} \hat{w}_h$ in the expansion in the eigenvectors of $A$. Since we do not assume that $R$ is unitary, $\|R^{-1}(\hat{w}^{(n+1)}_h - \hat{w}^\star_h)\|$ may greatly differ from the actual error $\|\hat{w}^{(n+1)}_h - \hat{w}^\star_h\|$ in the transient stages of the iteration. This could occur 
for large condition numbers of $R$, for instance 
when the eigenvectors of $A$ are nearly linearly dependent. 

\end{itemize}
Next, we discuss what the theorem implies for the convergence rate measured by the number of 
iterations $N_{\rm iter}$ needed to obtain a prescribed relative error $\tau$.
We have the following corollary about how $N_{\rm iter}$
is related to the eigenvalue gap $\varepsilon$.
See the numerical examples.

\begin{corollary}\label{cor:iter}
Suppose $\hat{w}^{(0)}_h=0$.
For the relative error in the eigen-expansion coefficients
\mbox{$R^{-1}(\hat{w}^{(n)}_h - \hat{w}_h^\star)$}, we have
$$
\frac{\|R^{-1}(\hat{w}^{(n)}_h - \hat{w}_h^\star)\|}{
\|R^{-1}\hat{w}_h^\star\|} \leq \tau,\qquad n\geq N_{\rm iter},
$$
where $N_{\rm iter}\leq |\log\tau|/\varepsilon$.
For the relative error in the solution, we have
$$
\frac{\|\hat{w}^{(n)}_h - \hat{w}_h^\star\|}{
\|\hat{w}_h^\star\|} \leq \tau,\qquad n\geq N_{\rm iter},
$$
where $N_{\rm iter}\leq (|\log\tau| + \log\kappa(R))/\varepsilon$.
\end{corollary}
\begin{proof}
This follows directly from \eqref{eq:oneiterconv}
and \eqref{eq:manyiter}, together with the fact that
$r^n\leq (1-\varepsilon)^n\leq e^{-\varepsilon n}$.
\end{proof}

It is of particular interest to understand how $N_{\rm iter}$
depends on $\omega$ in practice,
for different types of problems and discretizations.
We note first that by Corollary~\ref{cor:iter} it is at worst proportional to $1/\varepsilon$.
The theorem thus reduces the question to the study of
$\varepsilon$, which
depends on
how close the scaled eigenvalues $\lambda_j/\omega$ are to $\pm i$, measured
with the
distance function $d$ in \eqref{eq: distance def}.
The eigenvalues, in turn, depend in a non-trivial way on the
problem that is solved.
The relation between $N_{\rm iter}$ and $\omega$ is therefore
complicated. It is
strongly influenced by the properties of the problem that is solved: 
boundary conditions, geometry, wave speed $c$ and discretization parameter $h$. (Note that
the size of the matrix $A$ grows when $h\to 0$.)
We will give some general comments about it
here
but leave for future studies the precise details and conclusions.
We limit ourselves to two canonical types of problems set on
a bounded smooth domain $\Omega$.

\subsubsection*{Dirichlet/Neumann Boundary Conditions}
In this case the wave energy is preserved in the wave equation.
The first order continuous operators in \eqref{eq:firstorderops}
both 
have point spectrums.
Their eigenvalues lie on the imaginary axis
and their eigenfunctions form an orthonormal basis for $L^2(\Omega)$.
For discretizations that discretely preserve energy
the eigenvalues of the matrix $A$ are also purely imaginary
and converge to
those of the continuous operator.
This means that if waves at the dominant wave length are well resolved,
the convergence rate will depend
just on $\omega$ and the eigenvalues of the continuous
operator, not on the discretization parameter $h$.
Moreover, $A$ is typically normal in this
case, with unit condition number $\kappa(R)=1$.

Since $d(\lambda_j/\omega,\pm i)=\alpha|\lambda_j/\omega\pm i|^2$
in this case, the theorem gives
$$
  \varepsilon = \alpha\min_j
  \left|\frac{\lambda_j\pm i\omega}{\omega}\right|^2,
$$
i.e. $\varepsilon$ is quadratic in the the euclidean distance.
We have assumed that $\lambda_j\neq \pm i\omega$ for all $j$,
so $\varepsilon>0$, but we can in principle have eigenvalues arbitrarily close
to $\pm i\omega$, giving an arbitrary slow convergence.
If we order the eigenvalues along the imaginary axis 
we can define a set of uniformly non-resonant $\omega$
such that $\eta| \lambda_{j+1}-\lambda_j|\leq |i\omega-\lambda_{j}|\leq(1-\eta)| \lambda_{j+1}-\lambda_j|$
for some $j\in{\mathbb Z}$ and a fixed $\eta\in(0,1/2)$.
These frequencies thus lie in the interior of the
intervals $[\lambda_{j+1},\lambda_{j}]/i$.
Then there is a $j^*=j^*(\omega)\in{\mathbb Z}$ such that
$$
\varepsilon 
 \geq \alpha\eta^2
  \left|\frac{\lambda_{j^*+1}-\lambda_{j^*}}{\omega}\right|^2.
$$   
Referring to Weyl's asymptotic estimate
of eigenvalues distributions, $\lambda_j\sim j^{1/d}$ 
in $d$ dimensions this leads to
$$
\lambda_{j+1}-\lambda_j
\sim j^{1/d}-(j+1)^{1/d}\sim j^{1/d-1}
\sim \lambda_j^{1-d}
\quad \Rightarrow\quad
   \varepsilon \sim 
   \left|\frac{\lambda_{j^*}^{1-d}}{\omega}\right|^2\sim  \left|\frac{\omega^{1-d}}{\omega}\right|^2\sim  \omega^{-2d}.
$$ 
Thus, when restricted to uniformly non-resonant $\omega$ 
the number of iterations scale as $N_{\rm iter}\sim\omega^{2d}$ in this case.
(Note that $\kappa(R)=1$.)
This was discussed in \cite{WaveHoltz}.

\subsubsection*{Impedance Boundary Conditions}
These boundary conditions model open problems where waves can propagate
out of the domain. Indeed,
impedance conditions can be seen as a simple approximation
of an absorbing boundary condition. 
The convergence theory is in general more complicated than the Dirichlet/Neumann case.
The continuous operators \eqref{eq:firstorderops} may or may not have
eigenvalues. If they have, the discrete operator $A$
will have eigenvalues converging to them as $h\to 0$.
However, $A$ will also have many
spurious eigenvalues, which do not converge.
Additionally, $A$ is not normal and the condition number 
$\kappa(R)$ is not one, but grows with finer discretizations.
Nevertheless, since energy leaves the
domain the problem is dissipative and
all eigenvalues, continuous as well as discrete, will have a 
strictly negative real part.

The distribution of eigenvalues in the left half-plane for different
open problems is a classical subject in scattering theory. The proximity of
eigenvalues 
to the imaginary axis is related to how much
wave energy can be trapped by the geometry of $\Omega$ or the wave speed $c$. 
The existence
of ``resonance free" regions near the imaginary axis,
i.e. regions without eigenvalues,
is a particularly
important question, going 
back to the work of e.g.
Morawetz \cite{Mo66}, Lax–Phillips \cite{LP68} and Ralston \cite{Ra69}.
This phenomenon is also
closely connected to the $\omega$-dependence
of the stability
constant in frequency-explicit solution estimates of Helmholtz; see e.g. discussions in 
\cite{MoiolaSpence19,GrahamEtAl2019}. 

If there are no trapped waves, there is
a resonance free strip such that
 $\Re \lambda_j\leq -\delta<0$ for all eigenvalues of the continuous operator.
The theorem then gives that
\begin{align*}
\varepsilon &=\min_j
|\Re \lambda_j/\omega|
+
\alpha 
  \left|\frac{\Im\lambda_j\pm \omega}{\omega}\right|^2
\geq \frac{\delta}{\omega}.
\end{align*}
Therefore, $N_{\rm iter}$ would be proportional to just $\omega$.
In the discrete case, the numerical tests in one dimension, Section \ref{sec:Num1d}, indicate
that $\omega \varepsilon$ depends weakly on $\omega$ under the scaling \eqref{eq:homegascale}, and
almost entirely on the diameter of the domain, which is proportional to the time it takes a wave to flow from one end of the domain to the other and then propagate out. 
Indeed, we see $O(\omega)$ iterations to convergence also in the discrete case.
 
When there are trapped waves the eigenvalues of the
continuous operators can come exponentially close
to the imaginary axis. In these cases we expect that the convergence will become similar to
the Dirichlet/Neumann case.

Note that in the impedance case we must also account for
the effect of the condition number $\kappa(R)$, which may grow with
$\omega$ when the scaling 
\eqref{eq:homegascale} is used.
Our numerical experiments in Section \ref{sec:Num1d}
indicates that a polynomial bound on $\kappa$ in terms of $\omega$
is not a restrictive assumption. By Corollary \ref{cor:iter}
this suggests that a term $\sim \frac1\varepsilon\log\omega$,
which is just
logarithmic in $\omega$, would be added to
 $N_{\rm iter}$.

\subsubsection*{Extension}
In the more general case, when $A$ cannot be diagonalized, 
we can still show that the iteration converges and the error satisfies
\[ 
\|\hat{w}_h^{(n)} - \hat{w}_h^\star\| \leq K r^n \|\hat{w}_h^{(0)} - \hat{w}_h^\star\|, 
\]
for some $K > 0$.
This error bound is less precise than in Theorem \ref{maintheorem} as not much can be said about the constant $K$ in general. Nevertheless, the generalization to non-diagonalizable discretizations is important because upwind type discretizations are often not diagonalizable. Consider 
for example the most basic upwind finite difference discretization of $w_t + w_x = 0$, i.e. 
\[ 
\frac{dw_j}{dt} + \frac{w_j - w_{j-1}}{h} = 0. 
\]
For this discretization the matrix $-hA$ is bi-diagonal with ones on the main diagonal and negative ones on the first lower diagonal. The eigenvalues of $A$ are all $-h^{-1}$, but $A$ has only the single eigenvector $(0,\dots,0,1)^T$ and thus is not diagonalizable. Other upwind discretizations like the discontinuous Galerkin method with upwind flux are also not diagonalizable. In Section \ref{sec:num2d} we will see that DG with upwind flux produce very similar iterations with WaveHoltz as, for example, the diagonalizable finite difference scheme. This suggests that a different analysis not based on diagonalization and bounds on the condition number should be possible. Such analysis is left for the future.

\subsection{Proof of Theorem~\ref{maintheorem}}
In this section, we prove Theorem \ref{maintheorem} relying on Lemma \ref{lemma: beta complex} provided in Section \ref{sec:lemmas}.
To determine how the iterates $\hat{w}_h^{(n)}$ converge to $\hat{w}_h^\star$, we analyze the error vector at the $n$-th iteration $\hat{e}_h^{(n)} := \hat{w}_h^{(n)} - \hat{w}_h^\star$.
Substituting $\hat{w}_h^{(n)} = \hat{w}_h^\star + \hat{e}_h^{(n)}$ into \eqref{eq:semi-pi} and using \eqref{eq:pi-affine} we find
\begin{equation} \label{eq: error FPI}
    \hat{e}^{(n+1)}_h = \mathcal{S}_h \hat{e}_h^{(n)} \implies \hat{e}_h^{(n)} = (\Sw_h)^n \hat{e}_h^{(0)}.
\end{equation}
We will therefore have convergence $\hat{e}_h^{(n)}\to 0$ if 
$\rho(\Sw_h) < 1$, which is what we now prove.

In the definition of $\Sw_h\hat{w}$, we have
\[ 
\Sw_h \hat{w} = \frac{2}{T}\int_0^T \left(\cos(\omega t) - \frac{1}{4}\right) w(t) \, dt,
\qquad \frac{dw}{dt}=Aw, \quad w(0)=\hat{w}.
 \]
We can solve the ODE explicitly as $w(t) = e^{tA}\hat{w}$, giving
 \[ 
 \Sw_h \hat{w} = \left[\frac{2}{T}\int_0^T \left(\cos(\omega t) - \frac{1}{4}\right) e^{tA} \, dt\right] \hat{w} 
 \quad\Rightarrow\quad
\Sw_h = \frac{2}{T}\int_0^T \left(\cos(\omega t) - \frac{1}{4}\right) e^{tA} \, dt. 
\]
Since we have assumed that $A=R\Lambda R^{-1}$, we can simplify
further and also diagonalize $\Sw_h$ as
\begin{equation} \label{eq: diagonalization of S}
\Sw_h = R \left(\frac{2}{T}\int_0^T \left(\cos(\omega t) - \frac{1}{4}\right) e^{t\Lambda} \, dt\right) R^{-1}
=: RSR^{-1},
\end{equation}
where $S$ is a diagonal matrix with 
the eigenvalues $\{\mu_j\}$ of $\Sw_h$ as
entries. 
Substituting \eqref{eq: diagonalization of S} into \eqref{eq: error FPI} we have,
\[ 
\hat{{e}}_h^{(n+1)} = R S R^{-1}\hat{{e}}_h^{(n)}, \qquad\text{and} \qquad
\hat{{e}}_h^{(n)} = R S^n R^{-1} \hat{{e}}_h^{(0)}. 
\]
To both equations, we multiply both sides by $R^{-1}$, take the norm, and use the sub-multiplicative property of the matrix norm to conclude that
\[
\|R^{-1} \hat{e}_h^{(n+1)}\| \leq r\, \|R^{-1} \hat{e}_h^{(0)}\|, \qquad\text{and} \qquad 
\|\hat{e}_h^{(n)}\| \leq r^n\|R\| \cdot \| R^{-1} \| \cdot \|\hat{e}_h^{(0)}\|,
\]
where we have defined $r =\|S\|$.
If we can show 
(\ref{eq:rassumption}), we have thus proved that
the iteration converges, as well as (\ref{eq:oneiterconv})
and (\ref{eq:manyiter}).
Since $S$ is diagonal and $\|\cdot\|$ is a $p$-norm,
$$
r = ||S||=\rho(S)= \max_j|S_{jj}|=\max_j|\mu_{j}|.
$$
We note that for $j=1, \ldots, m$,
\[ \mu_j=
\frac{2}{T}\int_0^T \left(\cos(\omega t) - \frac{1}{4}\right) e^{t\lambda_j} \, dt
=\frac{1}{\pi}\int_0^{2\pi} \left(\cos(s) - \frac{1}{4}\right) e^{s\lambda_j/\omega} \, ds= \hat\beta\left(\frac{\lambda_j}{\omega}\right). \]
Therefore, by Lemma~\ref{lemma: beta complex} and 
by the definition of $\varepsilon$ above,
\begin{align*}
\max_j |\mu_j| &= \max_j |\hat\beta(\lambda_j/\omega)| \\
&\leq
\max_j \max\left\{1-d\left(\frac{\lambda_j}{\omega},i\right),
1-d\left(\frac{\lambda_j}{\omega},-i\right),
1-\delta\right\}
\leq
\max\left\{ 1-\varepsilon, 1-\delta \right\}.
\end{align*}
This proves (\ref{eq:rassumption}) and concludes the proof.

\subsection{Supporting Lemmas} \label{sec:lemmas}
The proof of Theorem \ref{maintheorem} relies on 
estimates of the scaled filter transfer function 
$\hat{\beta}(z)$ defined in \eqref{eq: beta def}.
The first estimate was proved in a previous paper and
considers only purely imaginary $z$.

\begin{lemma} \label{lemma: prev WH}
For all $y \in \R$, the scaled filter transfer function $\hat\beta$ satisfies
\begin{align*}
    |\hat\beta(iy)| &\leq 1 - (y-1)^2, &\text{when } |y-1| \leq \frac{1}{2}, \\
    |\hat\beta(iy)| &\leq 1 - (y+1)^2, &\text{when } |y+1| \leq \frac{1}{2}, \\
    |\hat\beta(iy)| &\leq \frac{3}{4}, &\text{else}.
\end{align*}
\end{lemma}
This is a slightly sharper version of Lemma 2.2 in \cite{garcia2022extensions} and can be proved in the same way as presented there.

In the present paper we need to generalize Lemma~\ref{lemma: prev WH}
to complex $z\in\Cminus$.
Since the decay of $|\hat\beta(z)|$ away from $z=\pm i$ is
not the same in all directions, we introduce the parabolic
distance (\ref{eq: distance def}). We then
get a concise description of the decay that mirrors the
result in Lemma~\ref{lemma: prev WH}.
\begin{lemma} \label{lemma: beta complex}
    The scaled filter transfer function $\hat\beta$
    is entire and
    maps the left
    half plane $\Cminus$ to
    the unit disc ${\mathbb D}$.
    There are universal constants $d_0>0$ and $\delta\in(0,1)$ such that 
for all $z \in \Cminus$,
\begin{align}\label{eq: beta estimates}
    |\hat\beta(z)| &\leq 1 - d(z,i), &\text{when } d(z,i) \leq d_0, \nonumber\\
    |\hat\beta(z)| &\leq 1 - d(z,-i), &\text{when } d(z,-i) \leq d_0, \\
    |\hat\beta(z)| &\leq 1-\delta, &\text{else}.\nonumber
\end{align}
Moreover, $\hat\beta$
can be expressed as
\begin{equation} \label{eq: beta hat}
    \hat\beta(z) = \begin{cases}
-\frac{1}{2}, & z = 0, \\
1, & z = \pm i, \\
\frac{3z^2 - 1}{4\pi z (z^2 + 1)} \left( e^{2\pi z} - 1 \right), & \mathrm{else}.
\end{cases}
\end{equation}
\end{lemma}
\begin{proof}
The expression
\eqref{eq: beta hat}
is given by a direct computation.
Since
$\hat{\beta}$ is
the Fourier transform
of an $L^2$ function compactly supported in ${\mathbb R}$,
it is entire by
the Paley-Wiener Theorem.

To show that
$\hat\beta$
maps $\Cminus$
to ${\mathbb D}$
we need to verify that
$|\hat\beta(z)|\leq 1$ whenever
$z\in \Cminus$.
In \eqref{eq: beta hat}, observe that $\hat{\beta}$ is a rational function multiplied by $e^{2\pi z} - 1$. The denominator of the rational function is of higher order than the numerator, so this function tends to zero uniformly at infinity. Thus, since $e^{2\pi z} - 1$ is bounded for all $z\in \Cminus$, $\hat{\beta}(z)\to 0$ uniformly as $|z|\to \infty$ for all $z\in \Cminus$. Therefore, there exists a $R_0\geq 2$ such that
    $|\hat{\beta}(z)| \leq 3/4$
    for
$|z| \geq R_0$
and $z\in\Cminus$.
Moreover, by Lemma~\ref{lemma: prev WH}
we have
$|\hat{\beta}(z)| \leq 1$
for all $z$
on the imaginary axis.
Hence,
outside and on the boundaries
of the domain
$$
B_0 := \{z\in \Cminus: |z|\leq R_0\},
$$
we have
$|\hat{\beta}(z)| \leq 1$
and by the maximum modulus principle
this holds also inside the domain,
because $\hat\beta$
is analytic.
We have thus shown that
$\hat\beta:\Cminus\mapsto{\mathbb D}$.

For the estimates in \eqref{eq: beta estimates}, we consider the domains
\[
G_{\varepsilon}=\{z\in \Cminus\ :\ d(z,i)\leq \varepsilon\},\qquad \varepsilon\leq \alpha,
\]
which we parameterize as
\[
z(s,t) =-s^2(1-t^2)+i\left(1+ \frac{1}{\sqrt\alpha} t s \right).
\]
Note that
$z(s,t)$ maps $[0,\sqrt{\varepsilon}]\times [-1,1]$ 
to $G_{\varepsilon}$ and that by the choice of the parameterization
\begin{equation} \label{eq: dvs}
  d(z(t,s),i)=s^2.
\end{equation}
Since $\hat\beta$ is analytic and $z$ a polynomial, we can
Taylor expand
$|\hat\beta(z(s,t))|^2$
in $s$ around $s=0$,
\[
 |\hat\beta(z(s,t))|^2 = b_0(t) + b_1(t)s +b_2(t) s^2 + R(s,t)s^3,
\]
where $R$ is a smooth function. To compute the $b$-coefficients,
we first compute derivatives of $\hat\beta$. 
From its definition,
\[
 \hat\beta^{(k)}(i)
=\frac{1}{\pi} \int_0^{2\pi} \left(\cos(s) - \frac{1}{4}\right) s^ke^{i s} \, ds,
\]
which are standard integrals. 
In particular $\hat\beta(i)=1$, $\hat\beta'(i)=\pi$
and $\hat\beta''(i)=(8\pi^2-3)/6$. Then since $z(0,t)=i$,
\begin{align*}
b_0(t) & = |\hat\beta(i)|^2 =1,\\
b_1(t) & = 2\Re \overline{\hat\beta(i)}\hat\beta'(i)z_s(0,t)
=2\Re \pi \frac{1}{\sqrt\alpha_1}it = 0, \\
b_2(t) & = \Re \left(
\overline{\hat\beta'(i)z_s(0,t)}\hat\beta'(i)z_s(0,t)+
\overline{\hat\beta(i)}\hat\beta''(i)z_s(0,t)^2+
\overline{\hat\beta'(i)}\hat\beta(i)z_{ss}(0,t)\right)\\
&=
\frac{\pi^2 }{\alpha}t^2
-\frac{8\pi^2-3}{6\alpha}t^2
-2\pi(1-t^2)
=
-2\pi + t^2\left(2\pi-
\frac{2\pi^2-3}{6\alpha}
\right).
\end{align*}
Hence, upon choosing $\alpha$
as in (\ref{eq: distance def}) we get $b_2=-2\pi$, 
and $b_0$, $b_1$, $b_2$ are all
independent of $t$. Moreover,
in view of (\ref{eq: dvs}),
\[
 |\hat\beta(z(s,t))|^2 =  1-2\pi d(z(t,s),i)+ R(s,t) d(z(t,s),i)^{3/2}.
\]
We now define
\[
  M = \max_{\stackrel{0\leq s\leq \sqrt{\alpha}}{|t|\leq 1}} |R(s,t)|, \qquad
  d_0 = \min\{ ((2\pi-2)/M)^{2} ,\, \alpha\},
\]
so that whenever $z\in G_{d_0}\subset G_{\alpha}$, due to \eqref{eq: dvs}, 
\[
  |\hat\beta(z)|^2 \leq 1-2\pi d(z,i)+d(z,i)^{3/2}M  \leq 
  1-2\pi d(z,i)+\sqrt{d_0} Md(z,i)  \leq 1-2d(z,i).
\]
After taking square roots on both side, this shows the first inequality in (\ref{eq: beta estimates}). The second follows
from the symmetry of $\hat\beta$ with respect to the real axis, $\hat\beta(\overline{z})
=\overline{\hat\beta(z)}$.

We have left to show the last inequality in (\ref{eq: beta estimates}).
However, before doing that we need to establish the geometry
of $G_{d_0}$ more carefully.
Suppose $z\in G_{d_0}$. Then
$$
 \alpha\geq d_0\geq d(z,i)\geq \alpha(\Im z-1)^2
 \quad\Rightarrow\quad 0\leq \Im z\leq 2,
$$
and as $|\Re z|\leq d(z,i)\leq d_0\leq \alpha \leq 1/\alpha$,
$$
 |z|^2 = |\Re z|^2 +|\Im z-1|^2 +2\Im z -1\leq
 \frac{1}{\alpha}|\Re z| +|\Im z-1|^2+3
 =\frac{1}{\alpha}d(z,i)+3
 \leq \frac{d_0}{\alpha}+3
 \leq 4.
$$
Therefore, as $R_0\geq 2$,
it holds that
$$
G_{d_0} \subset B_0 \cap
\{z\in \Cminus: \Im z\geq 0\}.
$$
We also define the corresponding set for the lower half plane,
namely $G'_{d_0} = \{z\in \Cminus\ |\ d(z,-i)\leq d_0\}$.
By the same arguments as above, 
$$
G'_{d_0} \subset B_0 \cap
\{z\in \Cminus: \Im z\leq 0\}.
$$
The sets are illustrated in Figure \ref{fig:D}.
    \begin{figure}[htbp]
        \centering
        \includegraphics[width=2in]{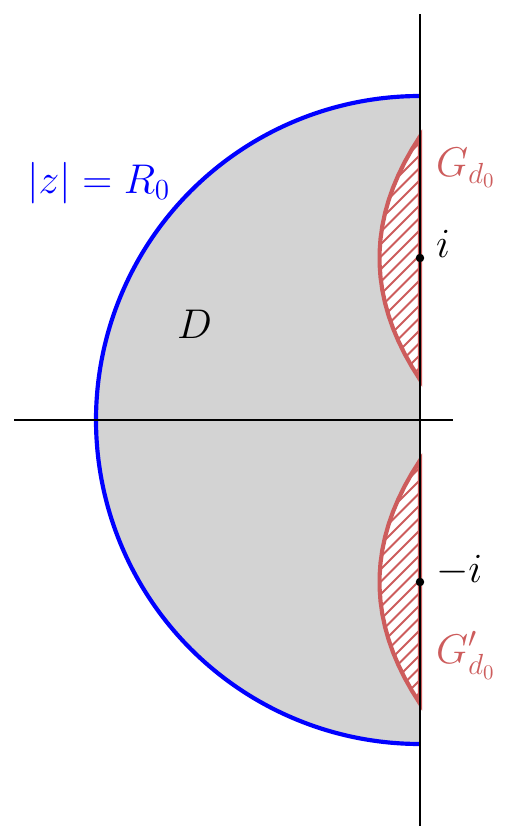}
        \caption{The regions $D$, $G_{d_0}$, and $G'_{d_0}$,
for some $d_0 < \alpha$.}
        \label{fig:D}
    \end{figure}
We need to prove that
$|\hat\beta(z)|$ is uniformly strictly smaller than one
outside $G_{d_0}\cup G'_{d_0}$.
We have already shown that outside $B_0$ we
have $|\hat\beta(z)|\leq 3/4$.
Define the remaining compact region $D$ in $\Cminus$ 
as
    \[ D = B_0 \setminus (G_{d_0} \cup G_{d_0}').  \]
We again use the maximum modulus principle and the analyticity of $\hat\beta$,
to estimate
$
  \max_{z\in D}|\hat\beta(z)|
  = 
  \max_{z\in \partial D}|\hat\beta(z)|.
$
Checking the maximum on each part of $\partial D$ we get 
first that $|\hat\beta(z)| \leq \frac{3}4$ on the outer arc $|z|=R_0$.
On the outer boundaries of $G_{d_0}$
and $G'_{d_0}$
we have
$|\hat\beta(z)| \leq 1-d_0$.
Finally, on the remaining part of the 
imaginary axis, where
$d(iy,\pm i)\geq d_0$,
Lemma \ref{lemma: prev WH} tells us that
\begin{align*}
|\hat\beta(iy)|  
&\leq
\max_{\stackrel{\alpha|y-1|^2\geq d_0\text{ and }\alpha|y+1|^2\geq d_0}{|y|\leq R_0}}
\max\{ 1-(y-1)^2 ,\, 1-(y+1)^2 ,\, 3/4 \}\\
&\leq
\max\{ 1-d_0/\alpha ,\, 3/4 \}
\leq 
\max\{ 1-d_0 ,\, 3/4 \}.
\end{align*}
Thus, in summary we can take $\delta = 1-\max\{1-d_0 , 3/4\}=\min\{d_0 , 1/4\}$ which concludes the proof.
\end{proof}

\section{Numerical Experiments} \label{sec:numeric}
In this section, we will empirically investigate the convergence of the WaveHoltz iteration for solving the  Helmholtz equation
\begin{subequations} \label{eq: basic Helmholtz}
    \begin{align}
        \Delta \hat{u} + \omega^2 \hat{u} = f(x), && x&\in\Omega, \\
        \mathbf{n}\cdot\nabla u = 0, && x&\in \Gamma_N, \\
        i\omega u + \mathbf{n}\cdot\nabla u = 0, && x&\in \Gamma_I.
    \end{align}
\end{subequations}
Here $\Omega$ is a bounded open subset of $\R^d$ for $d = 1$ or $d = 2$. The boundary of $\Omega$ is $\Gamma = \Gamma_N \cup \Gamma_I$ where $\Gamma_N$ is a subset of the boundary with Neumann boundary conditions and $\Gamma_I$ is a subset of the boundary with impedance boundary conditions which approximate absorbing boundary conditions on the bounded domain.
Since $\Gamma_I$ approximates an open boundary, this boundary will correspond to an outflow condition for the wave equation in the WaveHoltz iteration.
Some of the energy of the wave inside $\Omega$ will leak out of the domain over time at this boundary, so any convergent spatial discretization of the wave equation, which we write as the (real valued) system of differential equations
\begin{equation} \label{eq:wave_semiD}
    \frac{d}{dt}w_h = A w_h - F \cos(\omega t), \qquad \text{or} \qquad \frac{d}{dt}w_h = A w_h - F \sin(\omega t),
\end{equation}
must be dissipative, that is, all the eigenvalues of $A$ must lie in $\Cminus$.

We will consider two discretizations for the wave equation: the finite diference method (FD) and the discontinuous Galerkin method (DG). In the one dimensional experiments we will relate the eigenvalues and eigenvectors of the discretization matrix $A$ to the convergence rate of WaveHoltz.
We do not consider the case where the eigenvalues are purely imaginary (corresponding to energy conserving problems) as it has already been studied in \cite{WaveHoltz}.

In addition to the fixed point iteration, in this section we also consider WaveHoltz accelerated by GMRES. Since the fixed point of the WaveHoltz iteration satisfies,
\[ \hat{w}_h^\star = \Pi_h \hat{w}_h^\star = \Sw_h \hat{w}_h^\star + \pi_{0,h}, \]
we can apply GMRES to the linear system
\[ (I - \Sw_h) \hat{w}_h^\star = \pi_{0,h}, \]
which can be interpreted as a preconditioned system for the Helmholtz equation as shown in Section \ref{sec:aside}. We will see that GMRES substantially reduces the number of iterations to convergence. In the one dimensional examples we specify a relative tolerance for both the fixed-point iteration and the GMRES-accelerated iteration of $10^{-8}$.

\subsection{Numerical Methods}
We now briefly describe the numerical discretizations we use in the numerical experiments. 
\subsubsection{Finite Difference Method} \label{sec:FDM}
For the FD method we discretize the wave equation in first-order form:
\[
u_t = v, \qquad
v_t = \Delta u - f(x) \cos(\omega t).
\]
We apply the WaveHoltz iteration directly to this real-valued system, and recover the complex-valued estimate of the Helmholtz solution at iteration $n$ with $\hat{u}_h \approx \hat{u}^{(n)}_h + \frac{1}{i\omega}\hat{v}^{(n)}_h$ (see \ref{apx:real-valued-wh}).
The boundary conditions for the wave equation that will correspond to Neumann and impedance are, respectively,
\begin{align*}
    \mathbf{n}\cdot\nabla u &= 0, && x\in\Gamma_N, \; t > 0, \\
    u_t + \mathbf{n}\cdot\nabla u = v + \mathbf{n}\cdot\nabla u &= 0,  && x\in\Gamma_I,\; t > 0,
\end{align*}
where $\mathbf{n}$ is the outward facing normal to the boundary.

In one dimension we discretize this problem in the computational domain $\Omega = (-1, 1)$.
We approximate the spatial derivatives using central differences on the uniform grid $x_j = -1 + j h, \, j=0,\dots,m$ (where $h = 2/m$) and introduce the grid functions $(u_j,\, v_j) \approx (u(x_j), \, v(x_j))$.
Normal derivatives at the boundaries are approximated using ghost points and centered finite differences. For example at  $x_0$ we would have $(u_1 - u_{-1})/ 2 h = 0$. Outflow conditions are approximated similarly, using the ghost nodes. For example at $x_{m}$ we would have $v_m + (u_{m+1} - u_{m-1})/ 2h  = 0$. The spatial semi-discretization of the first order system is
\begin{align*}
    \frac{d}{dt} u_j &= v_j, \\
    \frac{d}{dt} v_j &= \frac{u_{j-1} - 2u_j + u_{j+1}}{h^2} - f(x_j) \cos(\omega t), \qquad j = 0, \dots, m.
\end{align*}
This spatial discretization has error proportional to $h^2 \omega^3$ for the Helmholtz problem, so, when we vary $\omega$, we select $h$ according to the rule $h^2\omega^3 = \mathrm{const.}$ Note that in \eqref{eq:wave_semiD} we write $w_h = \left(u_0, \ldots, u_m, v_0, \ldots, v_m \right)^T$ and $F = \left( 0, \ldots, 0, f(x_0), \ldots, f(x_n) \right)^T$.
In two dimensions the spatial discretization is similarly defined by applying the one dimensional derivative approximations in each direction of the two dimensional grid.

Our analysis on the convergence of $\hat{w}_h^{(n)}$ to the exact solution of the Helmholtz equation assumed that the system of differential equations in \eqref{eq:wave_semiD} is solved exactly in time, however we will solve this equation using a fourth order Runge-Kutta method of \cite{Runge-Kutta4} with small $\Delta t$ so that the temporal errors minimally affect the convergence rate. Throughout this section, $\|\cdot\|$ is the 2-norm.

\subsubsection{Discontinuous Galerkin Method} \label{sec:DGM}
For the DG method we discretize the wave equation in conservative form:
\begin{subequations} \label{eq:WAVE_EQuation_cons}
\begin{align}
    p_t + \nabla \cdot \mathbf{u} &= -\frac{1}{\omega}f(x)\sin(\omega t), \\
    \mathbf{u}_t + \nabla p &= 0.
\end{align}
\end{subequations}
We use the nodal DG-FEM on a mesh of elements $I_j, j=0,\dots,n$. 
We look for $(p_h, \mathbf{u}_h) \in C^1([0,T]; V)$ where $V$ is the usual DG polynomial approximation space such that for all $\varphi \in V$,
\begin{align*}
    ((p_h)_t, \varphi)_{L^2(I_j)} - a(\mathbf{u}_h, \varphi) &= -\frac{\sin(\omega t)}{\omega}(f, \varphi)_{L^2(I_j)}, \\
    ((\mathbf{u}_h)_t, \varphi)_{L^2(I_j)} - b(p_h, \varphi) &= 0.
\end{align*}
Here,
\begin{align*}
    a(\mathbf{v}, \varphi) &= \sum_{j=0}^n \int_{I_j} \mathbf{v} \cdot \nabla\varphi \, dx - \int_{\Gamma_j} (\mathbf{v} \cdot \mathbf{n})^\sharp \varphi \, ds \\
    b(q,\varphi) &= \sum_{j=0}^n \int_{I_j} q \nabla \varphi \, dx - \int_{\Gamma_j} q^\sharp \varphi\, \mathbf{n} ds.
\end{align*}
Here $(\mathbf{v} \cdot \mathbf{n})^\sharp$ and $q^\sharp$ are numerical fluxes for the quantities $\mathbf{v} \cdot \mathbf{n}$ and $q$, respectively, on the element interface $\Gamma_j = \partial I_j$. We take these fluxes to be the central/average flux:
\[ q^\sharp = \frac{1}{2} (q^- + q^+), \qquad (\mathbf{v} \cdot \mathbf{n})^\sharp = \frac{1}{2}(\mathbf{v}^- + \mathbf{v}^+)\cdot \mathbf{n}. \]
We choose the sub-optimal central flux (rather than the upwind flux) because the upwind flux produces a matrix which is not diagonalizable meaning we cannot examine the assumptions of Theorem \ref{maintheorem} in detail for the upwind flux discretization. However, in Section \ref{sec:num2d} we consider the upwind flux and find that it converges at roughly the same rate as the FD discretization.
The implementation is based on \cite{Kopriva-ImplementingSpectralMethods} and \cite{HesthavenWarburton-NodalDG}. We evolve the solution in time with the same fourth order Runge-Kutta method as for the FD method.

Here we can take $p$ and $\mathbf{u}$ to be real valued and recover the complex valued solution via the procedure in \ref{apx:real-valued-wh}. 

\subsection{Numerical Examples in One Dimension} \label{sec:Num1d}
Consider the problem:
\begin{subequations} \label{eq:1d-test-problem}
    \begin{gather}
        \frac{d^2\hat{u}}{dx^2} + \omega^2 \hat{u} = f, \qquad x\in (-1,1), \\
        \frac{d\hat{u}}{dx}(-1) = 0, \qquad i\omega \hat{u}(1) + \frac{d\hat{u}}{dx}(1) = 0.
    \end{gather}
\end{subequations}
In the following sections we will investigate the discretization matrices that arise from this problem via the finite difference and discontinuous Galerkin methods previously described. We will consider the eigenvalues and eigenvectors of these matrices for various values of $\omega$ and find that they satisfy the conditions of Theorem \ref{maintheorem}. We will then consider a test problem where we begin with a specially selected initial guess which has a slow rate of convergence at first, but eventually converges quickly.

\subsubsection{Results using the Finite Difference Method}
By Theorem \ref{maintheorem}, we expect WaveHoltz to converge with rate 
\[ r \leq \min\{ 1 - \varepsilon, 1-\delta\}, \qquad \varepsilon = \min_j d(\lambda_j/\omega, \pm i),\]
where $\lambda_j$ are the eigenvalues of the discretization matrix $A$. In this experiment we verify this convergence rate for $\omega$ between $10\pi$ and $30\pi$.

In order to be in the well resolved regime, we choose $h^2 \omega^3 = 10$. This will correspond to approximately 10 points per wavelength for $\omega = 10\pi$. First we verify \ref{assumption:E+U} by plotting the eigenvalues of $A$ scaled by $1/\omega$ for $\omega = 10\pi, 20\pi,$ and $30\pi$ in Figure \ref{fig:1d-eigs}. In Figure \ref{fig:1d-eigs} we see that indeed the parabolic distance to $i$ is not zero, and
as $\omega$ increases the scaled eigenvalues 
get closer to the imaginary axis. The black contours are level sets of the parabolic distance in Theorem \ref{maintheorem}.

\begin{figure}[]
    \centering
    \includegraphics[width=0.32\textwidth]{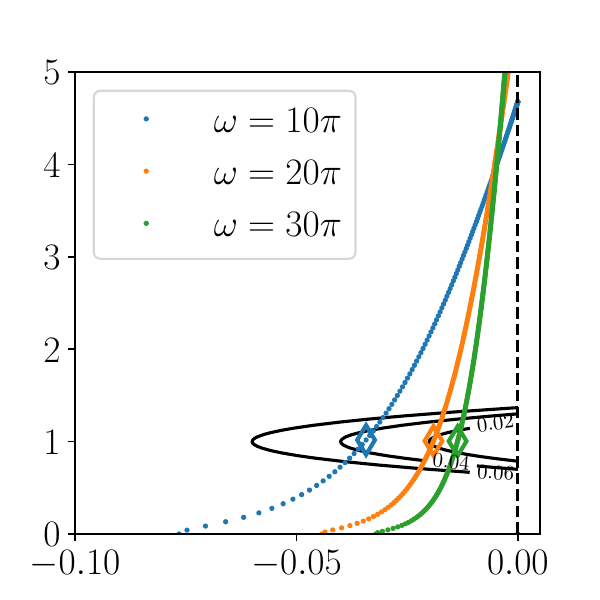}
    \includegraphics[width=0.32\textwidth]{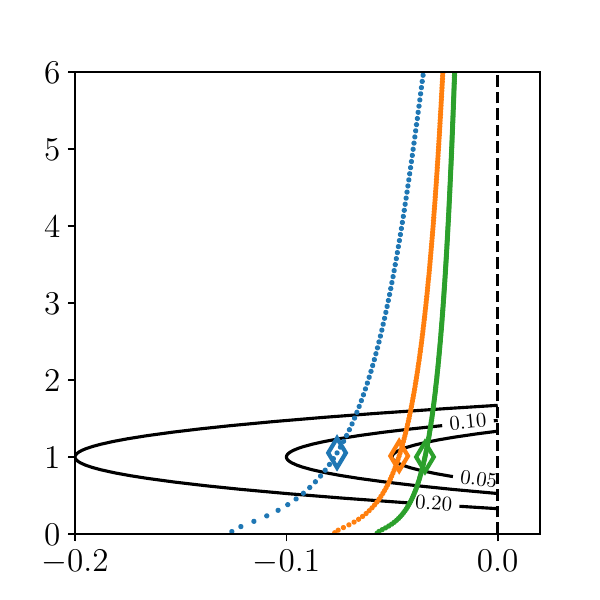}
    \includegraphics[width=0.32\textwidth]{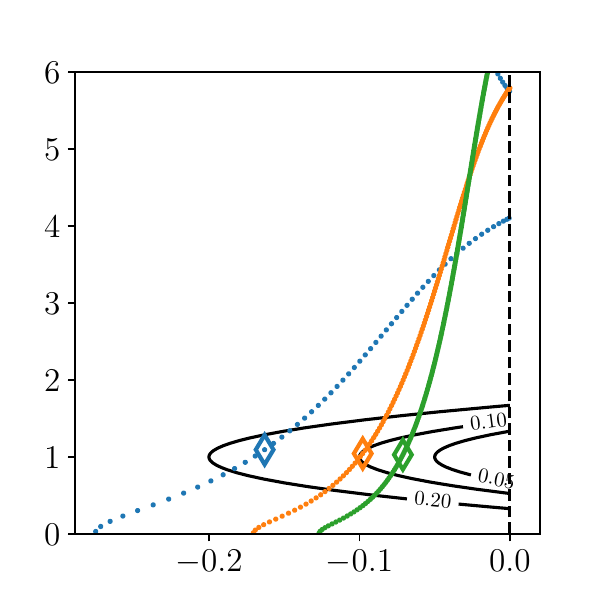}
    \caption{The eigenvalues of the discreization matrix $A$ scaled by $1/\omega$ for finte differences (left) and DG with $P=1$ (middle)  and $P=2$ (right). Because $A$ is real, all of its eigenvalues are either real or appear in complex conjugate pairs, so we plot the eigenvalues where $\Im\{\lambda\} \geq 0$. The diamonds highlight the eigenvalues minimizing the parabolic distance of which we draw three level sets. As $\omega$ increases we observe the distance between $i$ and the nearest scaled eigenvalue decreasing. Note that the imaginary part is much larger in magnitude than the real part.
    \label{fig:1d-eigs}}
\end{figure}

Turning now to the rate of convergence, in Figure \ref{fig:fd1d-eps-kappa} we plot $\varepsilon$, and $1 - \rho(\Sw_h)$ for various $\omega$. In the proof of Theorem \ref{maintheorem}, we showed that $\varepsilon \leq 1 - \rho(\Sw_h)$ which is verified in the figure.
For this discretization we see that $\varepsilon$ decreases sublinearly in terms of $\omega$ which suggests that WaveHoltz should converge quickly for this problem.

In the same plot we plot a typical estimator of the convergence rate
\[ \hat{\rho}(n) = \left( \frac{\hat{e}_h^{(n)}}{\hat{e}_h^{(0)}} \right)^{\frac{1}{n}}. \]
Note that $\hat{\rho}$ depends on the forcing and initial data and is bounded above by the theoretical convergence rate $\rho$.
The problem setup used here exemplifies a worst case scenario.
In Figure \ref{fig:fd1d-eps-kappa}, we report $\hat{\rho}$ after the last iteration (after which the error falls below a specified tolerance of $10^{-8}$) for the problem setup described below.
In the same figure, we also plot the condition number $\kappa(R)$ introduced in Theorem \ref{maintheorem}.
For this problem, it appears that the condition number scales roughly like $\kappa(R) \sim O(h^{-2}) = O(\omega^3)$.

\begin{figure}[]
    \centering
    \includegraphics[width=0.49\textwidth]{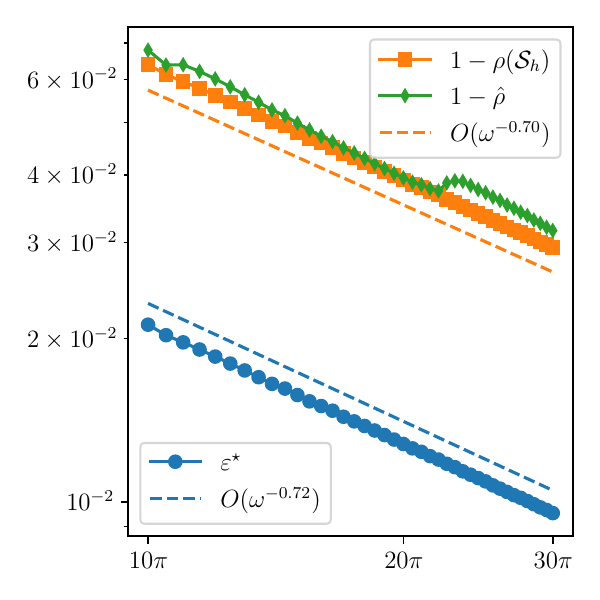}
    \includegraphics[width=0.49\textwidth]{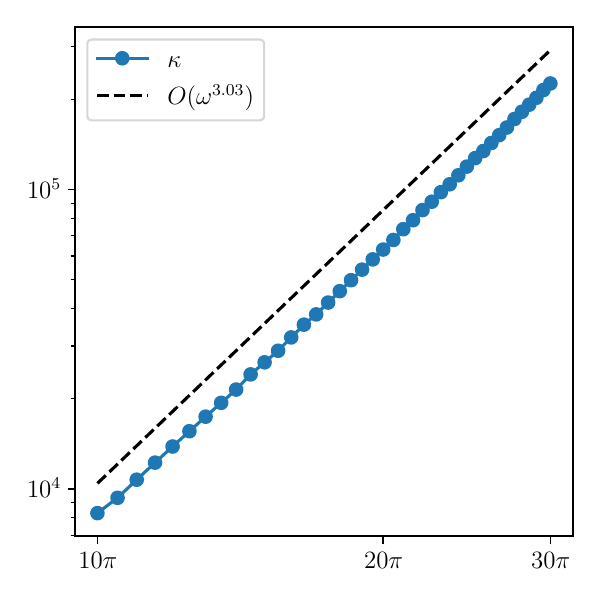}
    \caption{We display $\varepsilon$ vs. $\omega$ (left), and the condition number $\kappa(R)$ vs. $\omega$ (right). The slopes of the dashed lines for both plots were estimated by regression.
    \label{fig:fd1d-eps-kappa}}
\end{figure}

Next, we examine the number of iterations to convergence for a worst case problem. We will study this by examining the error $\hat{e}^{(n)}_h = \hat{w}^{(n)}_h - \hat{w}^\star_h$ which satisfies the fixed point iteration:
\[ \hat{e}_h^{(n+1)} = \Sw_h \hat{e}_h^{(n)} = \Sw_h^{n+1} \hat{e}_h^{(0)}. \]
We also monitor the norm of the error in the coefficients of the eigenvectors $c_h^{(n)}$ which are defined as
\[ c_h^{(n)} := R^{-1} \hat{e}_h^{(n)}. \]
We select an initial condition for the error as
\begin{gather}
    u_0(x) = 2\sin^2(\pi x)\sin(\omega x), \notag \\
    \hat{e}_h^{(0)} = (u_0, -\frac{d}{dx}u_0).  \label{eq: fd1d IC}
\end{gather}
The right-hand-side $f$ is implicitly defined in terms of this initial error. This choice of initial condition is selected because the fixed point iteration makes very slow progress towards the solution in the first few iterations, nevertheless, we find that eventually the estimates in Theorem \ref{maintheorem} match the observed behavior.

As suggested by Theorem \ref{maintheorem} we expect the WaveHoltz iteration to converge with rate no worse than $1-\varepsilon$. Since in Figure \ref{fig:fd1d-eps-kappa} we observed that $\varepsilon \sim O(\omega^{-0.72})$, hence for this problem we expect
\[ \frac{\|c_h^{(n+1)}\|}{\|c_h^{(n)}\|} \sim O(1 - \omega^{-0.72}), \]
and, since $\kappa \sim O(\omega^{3})$, we also expect
\[ \frac{\|\hat{e}_h^{(n)}\|}{\|\hat{e}_h^{(0)}\|} \sim O(\omega^{3}(1 - \omega^{-0.72})^n). \]

To the left in Figure \ref{fig:iters_v_rel_error} we plot the relative error as a function of the iteration count until it is below $10^{-8}$. We do this for $\omega = 10\pi, 20\pi,$ and $30\pi$. We see that the true error $\|\hat{e}_h\|$ may sometimes stagnate or even briefly increases, it may also decrease much faster than $\|c_h\|$, but on average they match closely.
In the same figure, we also plot the relative error of the GMRES-accelerated iteration which decrease much faster. To the right in Figure \ref{fig:iters_v_rel_error} we plot the number of iterations until the relative error falls below $10^{-8}$. We observe that the number of iterations to convergences scales better than $O(\omega)$ which is consistent with the discussion of impedance boundary conditions in Section \ref{sec:conv-theory}. For this problem, we observe that the GMRES-accelerated iteration requires roughly 3 times fewer iterations than the fixed-point iteration.

\begin{figure}[htbp]
    \centering
    \includegraphics[width=0.49\textwidth]{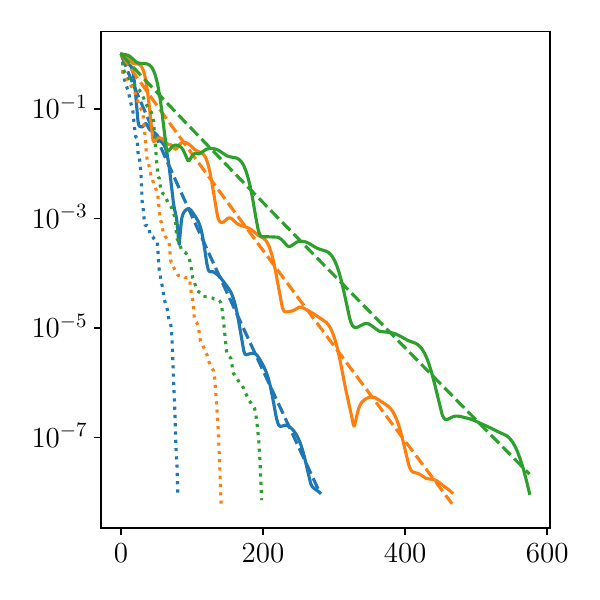}
    \includegraphics[width=0.49\textwidth]{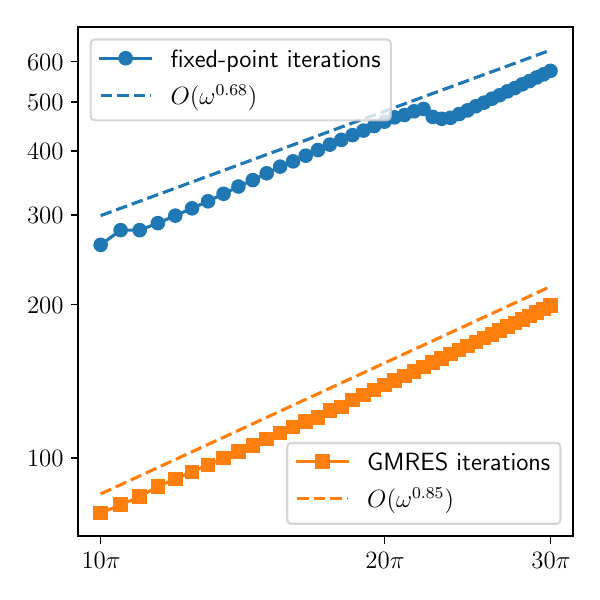}
    \caption{
    The left figure shows the relative error (solid), the eigenvector coefficient error (dashed), and the GMRES-accelerated error (dotted), with blue, orange, and green representing $\omega=10\pi, 20\pi, 30\pi$.
The right figure reports the number of iterations required for each method to achieve $\|\hat{e}_h^{(n)}\| \leq 10^{-8}\|\hat{e}_h^{(0)}\|$.
    }
    \label{fig:iters_v_rel_error}
\end{figure}

\subsubsection{Results using the Discontinuous Galerkin Method}
Consider the Helmholtz equation in \eqref{eq:1d-test-problem}. We discretize this problem with uniformly spaced elements and degrees of freedom specified on the Legendre-Gauss-Lobatto quadrature points. The order of the DG discretization is $P+\frac{1}{2}$ where $P$ is the degree of the polynomial basis functions on each element. Like in the finite difference discretization we choose our degrees of freedom according to $h^{P+\frac{1}{2}}\omega^{P+\frac{3}{2}} = 10.$

In Figure \ref{fig:1d-eigs} we plot the eigenvalues for discretizations of all combinations of $P=1,2$ and $\omega=10\pi, 20\pi, 30\pi$. As with the finite difference case, the eigenvalues are bounded away from $i\omega$, and since the real part of this distance is relatively large, we should expect fast convergence. Note that this distance is roughly two orders of magnitude larger than for the finite difference case. This difference is reflected in the convergence rates of these methods as DG converges much faster for this problem than FD. We speculate that this is a consequence of the accuracy of the impedance boundary conditions. In one dimension, the DG scheme's non-reflecting boundary conditions are exact to machine precision, whereas the finite difference boundary condition is $O(h^2)$, so it takes longer to dissipate the non-time-harmonic waves.

\begin{figure}[]
    \centering
    \includegraphics[width=0.49\textwidth]{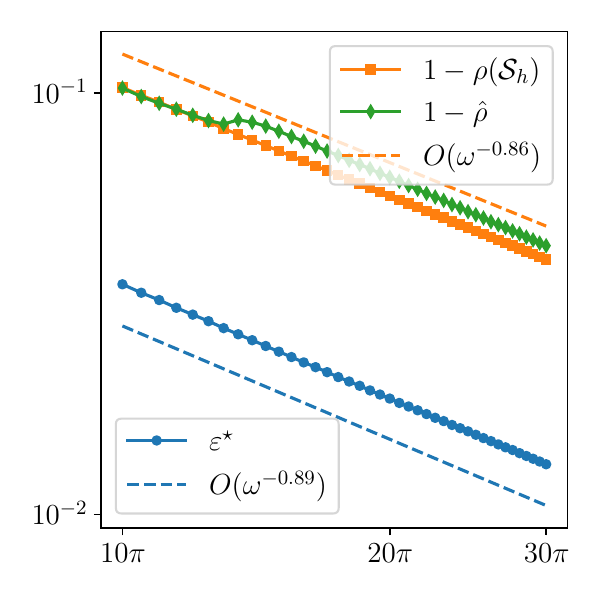}
    \includegraphics[width=0.49\textwidth]{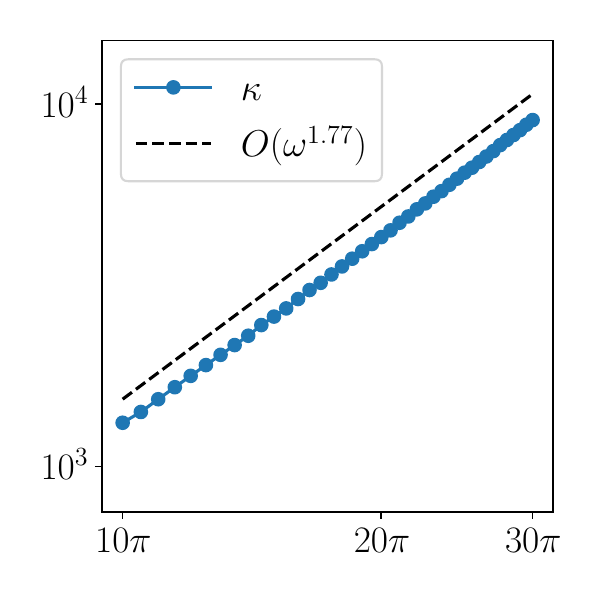} \\
    \includegraphics[width=0.49\textwidth]{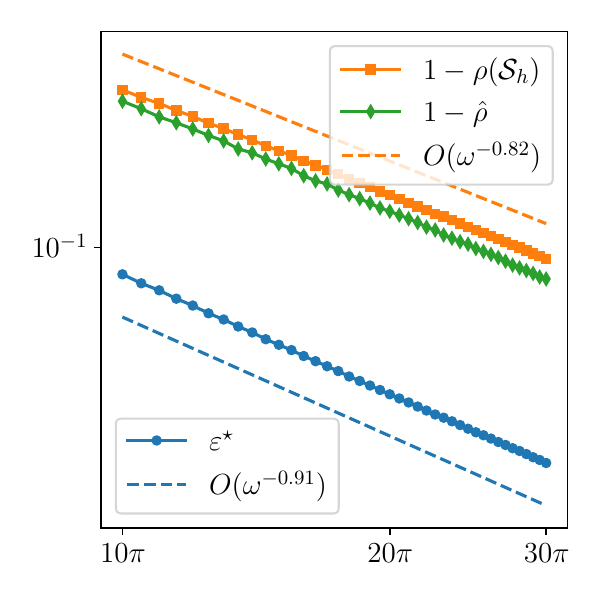}
    \includegraphics[width=0.49\textwidth]{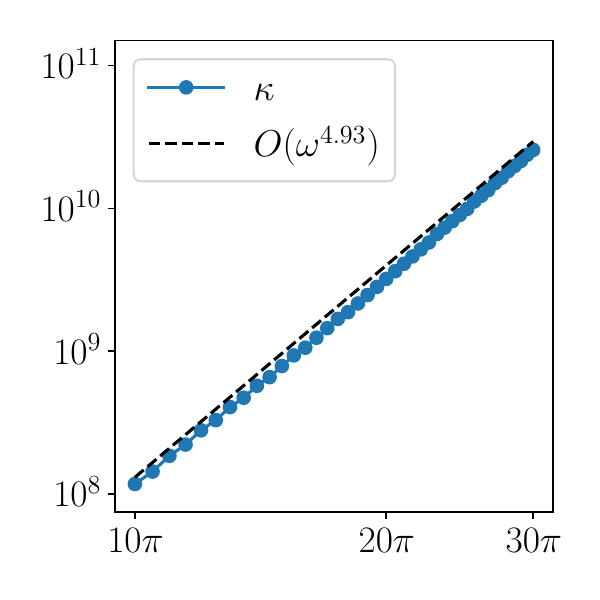}
    \caption{We display $\varepsilon$ vs. $\omega$ (left column), and condition number $\kappa(R)$ vs. $\omega$ (right column). To top row corresponds to $P = 1$, and the bottom row to $P = 2$.
    \label{fig:dg1d-kappa}}
\end{figure}

In Figure \ref{fig:dg1d-kappa} we plot $\varepsilon$, $1-\rho(\Sw_h)$, and $1-\hat{\rho}$ for uniformly spaced values of $\omega$. Additionally, we plot the condition number of the eigenvector matrix. We see that the precise convergence rate $\rho(\Sw_h)$ is slightly better than the estimate and gets closer to 1 sublinearly for both $P=1$ and $P=2$ as $\omega$ increases. The condition number $\kappa(R)$ appears to be bounded by a polynomial in $\omega$ for both $P=1$ and $P=2$. Despite the poor scaling of $\kappa(R)$ which from Figure \ref{fig:dg1d-kappa} can be overestimated as $O(\omega^5)$ for $P=2$, the convergence rate is still much faster that for the FD case as seen in Figure \ref{fig:dg1d-error}. For example, for $\omega = 10\pi$, the FD error reaches $10^{-8}$ in roughly 300 iterations where as the DG error takes roughly 200 iterations with $P=1$, and 100 iterations with $P=2$.

\begin{figure}[]
    \centering
    \includegraphics[width=0.49\textwidth]{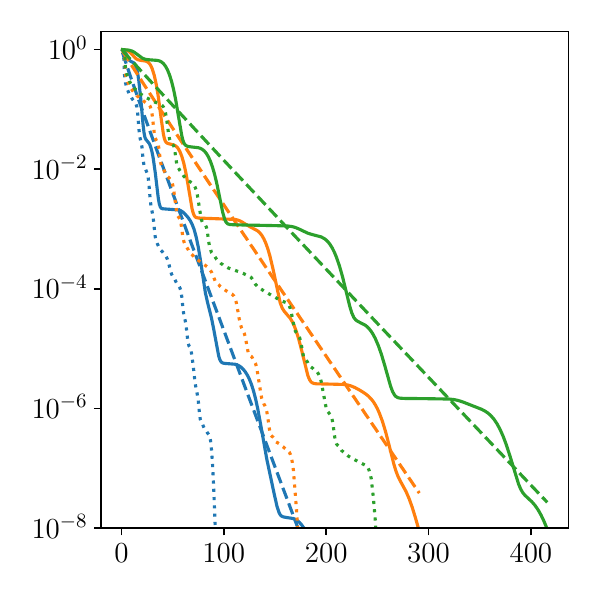}
    \includegraphics[width=0.49\textwidth]{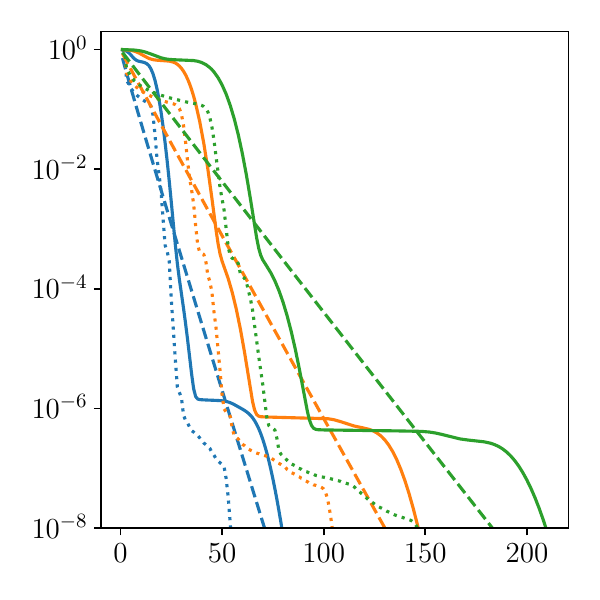}
    \caption{The figures show the relative error (solid), the eigenvector coefficient error (dashed), and the GMRES-accelerated error (dotted), with blue, orange, and green representing $\omega=10\pi, 20\pi, 30\pi$.
    \label{fig:dg1d-error}}
\end{figure}

In Figure \ref{fig:dg1d-error}, we also display the convergence rate in terms of $\hat{e}_h$, $c_h$, and error of the GMRES-accelerated iteration for the same initial condition \eqref{eq: fd1d IC} as for the FD example. We observe $P = 2$ converges roughly twice as fast as $P = 1$. For $P = 1$, GMRES appears to require roughly half as many iterations, while for $P = 2$, the acceleration is less noticeable, but still requiring fewer than $\approx 75\%$ the iterations of the fixed-point iteration.

\subsection{A Numerical Experiment in Two Dimensions} \label{sec:num2d}
We now consider the Helmholtz equation \eqref{eq: basic Helmholtz} in two dimensions
on the domain $\Omega = (-1,1)\times (-1,1)$.
We specify Neumann boundary conditions at $x = -1$ and $y = -1$ and outflow boundary conditions at $x = 1$ and $y = 1$. This problem setup does not support trapped waves. We place a ``point''-source forcing term centered at $(-0.7, -0.1)$
\[ f(x, y) = \pi^{-1}\omega^2 e^{-\omega^2[(x+0.7)^2 + (y+0.1)^2]}. \]
The modulus of the solution produced by the finite difference method is plotted in Figure \ref{fig:2d-sol} for $\omega=10\pi, 20\pi, $ and $30\pi$.

Due to the number of degrees of freedom needed to resolve the solution, computing the eigenvalues and eigenvectors is computationally very difficult for even moderate frequencies $\omega$. Instead, we will examine the iteration directly. 
We start by analyzing the convergence behaviour by looking at the relative residual of the iteration:
\[ \mathrm{res}^{(n)} = \frac{\|\hat{w}_h^{(n)} - \hat{w}_h^{(n-1)}\|}{\|\hat{w}_h^{(1)} - \hat{w}_h^{(0)}\|} = \frac{\|\hat{w}_h^{(n-1)} - \Pi_h\hat{w}_h^{(n-1)}\|}{\|\hat{w}_h^{(0)} - \Pi\hat{w}_h^{(0)}\|}. \]
The residual $\mathrm{res}^{(n)}$ is only an estimate of the relative error $\|\hat{e}_h^{(n)}\| / \|\hat{e}_h^{(0)}\|$ but is a more practical metric since we do not assume we have access to $\hat{w}^\star_h$ during the iteration.
For these examples, we specify a tolerance of $10^{-6}$.

\begin{figure}[]
   \centering
   \includegraphics[width=0.3\textwidth]{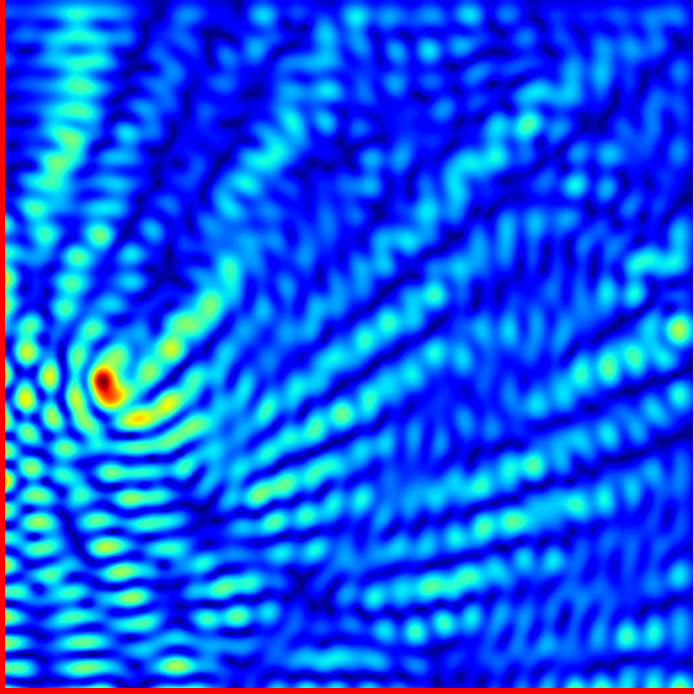}
   \includegraphics[width=0.3\textwidth]{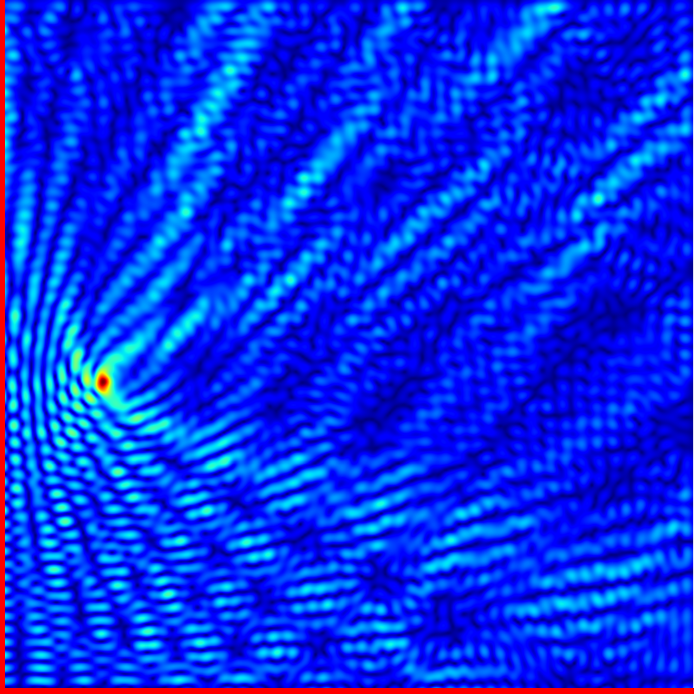}
   \includegraphics[width=0.3\textwidth]{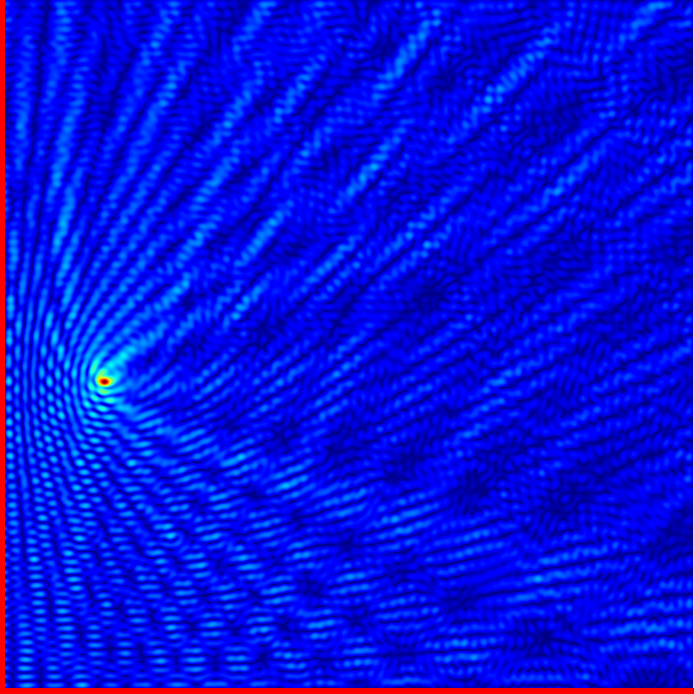}
   \caption{The modulus of the solution to the Helmholtz problem for $\omega = 10\pi, 20\pi,$ and $30\pi$ from left to right. The red boundaries are Neumann boundary conditions.
   \label{fig:2d-sol}}
\end{figure}

In Figure \ref{fig:2d-omega-iter}, we display results for uniformly spaced frequencies between $10\pi$ and $30\pi$, for both the FD and DG methods.
This time, for the DG scheme, instead of the central flux we use the upwind flux
\[
q^{\sharp} = \frac{1}{2}(q^{-} + q^{+}) + \frac{1}{2}\mathbf{n}\cdot(\mathbf{v}^{-} - \mathbf{v}^{+}), \qquad
(\mathbf{v} \cdot \mathbf{n})^{\sharp} = \frac{1}{2}\mathbf{n}\cdot(\mathbf{v}^{+} + \mathbf{v}^{-}) + \frac{1}{2}(q^{-} - q^{+}).
\]
The DG scheme with the upwind flux is not diagonalizable (in contrast with the central flux consider for the one dimensional example), and as pointed out in the discussion of Theorem \ref{maintheorem} in Section \ref{sec:conv-theory} it is therefore harder to estimate the convergence rates. Nevertheless for this example, the WaveHoltz iteration converges as would be expected with a diagonalizable scheme.

In particular, we plot the number of iterations to reach the specified relative residual of $10^{-6}$. For the FD method, the number of iterations $N\sim O(\omega^{0.79})$ in this range of frequencies which is consistent with the one dimensional test problem. For the two DG schemes, the number of iterations scales like $N\sim O(\omega)$ in this range of frequencies. This is slightly worse than observed in the one dimensional test problem but consistent with the discussion in Section \ref{sec:conv-theory}.
From Figure \ref{fig:2d-omega-iter}, we observe that the lower order $P=1$ converges faster than $P=2$. We speculate that this difference is because the upwind scheme is  fairly dissipative (more so for lower order) which further shifts the real part of the eigenvalues away from the real axis, and therefore allows faster convergence of the WaveHoltz iteration. 

Finally we note that in \cite{WaveHoltz} (Section 4.2.1) we presented numerical experiments illustrating the convergence of WaveHoltz accelerated with GMRES in trapping and non-trapping geometry. While the theory presented here is for WaveHoltz without GMRES acceleration, the results in \cite{WaveHoltz} follow the same pattern with the number of iterations scaling like $\mathcal{O}(\omega)$ or faster for non-trapping geometries, and $\mathcal{O}(\omega^p)$ with $1<p<2$ for trapping geometry.    

\begin{figure}[]
    \centering
    \includegraphics[width=0.49\textwidth]{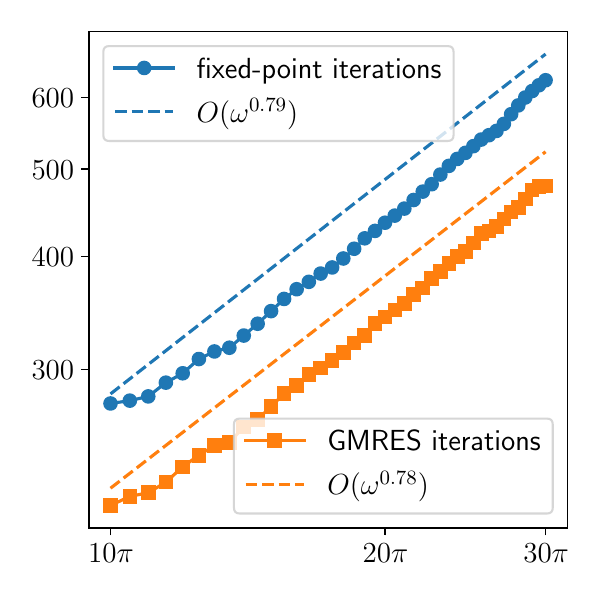}
    \includegraphics[width=0.49\textwidth]{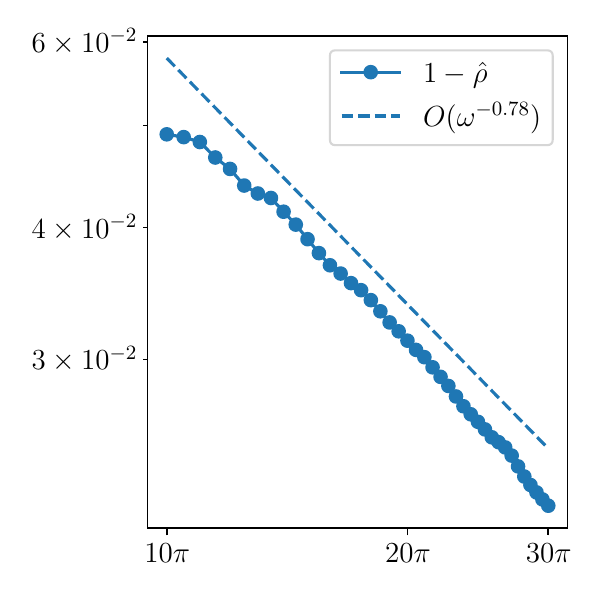} \\
    \includegraphics[width=0.49\textwidth]{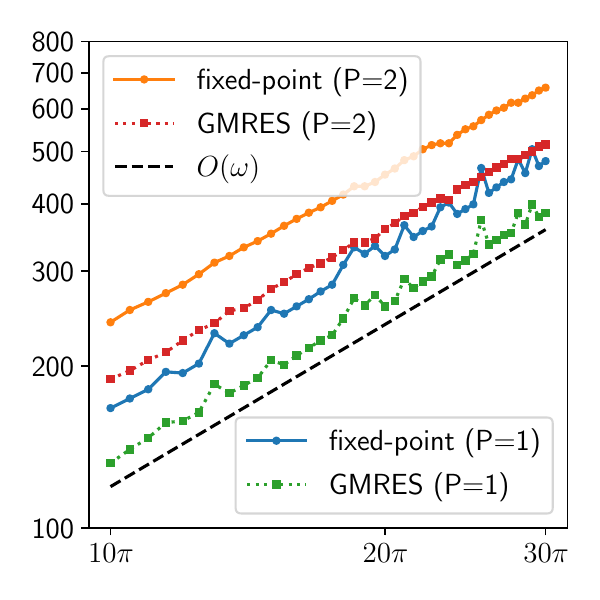}
    \includegraphics[width=0.49\textwidth]{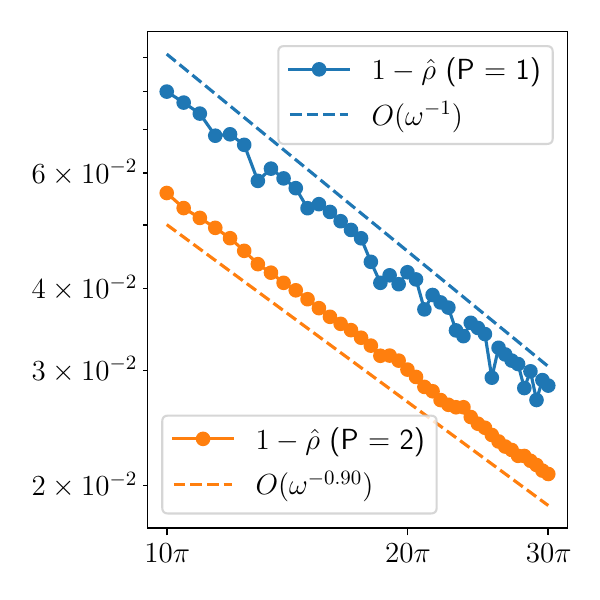}
    \caption{We display the number of iterations needed to reach a relative residual of $10^{-6}$ vs. $\omega$ (left column), and the estimated convergence rate $\hat{\rho}$ (right column). The top row corresponds to the FD discretization, and the bottom row to the DG discretization.}
    \label{fig:2d-omega-iter}
\end{figure}

\section{Conclusions}
We generalized previous results on the convergence of the WaveHoltz iteration to problems obtained by applying the WaveHoltz algorithm to the wave equation discretized in space. We proved that this semi-discrete WaveHoltz algorithm converges to an approximate solution of the Helmholtz equation for any stable discretization (regardless of boundary conditions). We estimated the convergence rate of the iteration as the (parabolic) distance of the eigenvalues of the spatial discretization matrix from $\pm i\omega$. This estimate was found to be modest compared to the true convergence rate for the numerical examples considered. We believe that tighter bounds on the convergence rate are possible, but this analysis is left to the future. We presented numerical examples verifying the assumptions of the main theorem. An immediate future extension of the results here are to consider the convergence rate of the WaveHoltz iteration when accelerated with Krylov methods and time discrete schemes. 

\section{Acknowledgments}
This material is based upon work supported by the National Science Foundation under Grant Number DMS-2345225 DMS-2436319 and Virginia Tech. (DA). Part of this material is based upon work supported by the National Science Foundation under Grant No. DMS-1928930 while DA was in residence at the Simons Laufer Mathematical Sciences Institute in Berkeley, California, during the Fall 2025 semester. Any opinions, findings, and conclusions or recommendations expressed in this material are those of the author(s) and do not necessarily reflect the views of the National Science Foundation.

\appendix
\section{WaveHoltz with Real Valued Wave Equation} \label{apx:real-valued-wh}
Consider the real valued wave equation in conservative form
\begin{subequations} \label{eq:WAVE_EQuation_cons appendix}
\begin{align}
    p_t + \nabla \cdot \mathbf{u} = -\frac{1}{\omega}f(x)\sin(\omega t), \ \
    \mathbf{u}_t + \nabla p = 0.
\end{align}
\end{subequations}
The corresponding complex valued Helmholtz equation is
\begin{subequations} \label{eq: helmholtz cons}
\begin{align}
    i\omega \hat{p} + \nabla \cdot \hat{\mathbf{u}} = -\frac{1}{i\omega}f(x), \ \
    i\omega \hat{\mathbf{u}} + \nabla \hat{p} = 0.
\end{align}
\end{subequations}
When eliminating $\hat{\mathbf{u}}$ we find
\[ \Delta \hat{p} + \omega^2 \hat{p} = f. \]
Which is the target Helmholtz equation.
Taking the real and imaginary parts of equation \ref{eq: helmholtz cons} results in the equations 
\begin{align*}
    &-\omega \Im\{\hat{p}\} + \Re\{\nabla\cdot\hat{\mathbf{u}}\} = 0, &
    \omega \Re\{\hat{p}\} + \Im\{\nabla\cdot\hat{\mathbf{u}}\} = \frac{1}{\omega}f, \\
    &-\omega \Im\{\hat{\mathbf{u}}\} + \Re\{\nabla \hat{p}\} = 0, &
    \omega \Re\{\hat{\mathbf{u}}\} + \Im\{\nabla \hat{p}\}  = 0.
\end{align*}
If we look for solutions to wave equation \eqref{eq:WAVE_EQuation_cons} of the form
\begin{align*}
    p &= p_0 \cos(\omega t) - p_1 \sin(\omega t), \\
    \mathbf{u} &= \mathbf{u}_0 \cos(\omega t) - \mathbf{u}_1 \sin(\omega t),
\end{align*}
then,
\begin{align*}
    -\omega p_0 \sin(\omega t) - \omega p_1 \cos(\omega t) + \nabla\cdot\mathbf{u}_0\cos(\omega t) - \nabla\cdot \mathbf{u}_1 \sin(\omega t) &= \frac{1}{\omega}f \sin(\omega t) \\
    -\omega \mathbf{u}_0 \sin(\omega t) - \omega \mathbf{u}_1 \cos(\omega t) + \nabla p_0 \cos(\omega t) - \nabla p_1 \sin(\omega t) &= 0.
\end{align*}
By matching the cosine and sine terms, we find that
\begin{align*}
   & -\omega p_1 + \nabla\cdot\mathbf{u}_0 = 0, & \omega p_0 + \nabla\cdot\mathbf{u}_1 = \frac{1}{\omega}f, \\
  &  -\omega \mathbf{u}_1 + \nabla p_0 = 0, & \omega \mathbf{u}_0 + \nabla p_1 = 0.
\end{align*}
These equations are the same as those for the real and imaginary parts of $\hat{p}$ and $\hat{\mathbf{u}}$, so we can conclude
\[ \hat{p} = p_0 + i p_1 \qquad \hat{\mathbf{u}} = \mathbf{u}_0 + i \mathbf{u}_1. \]
Moreover, we use the equations above to find
\[ p_1 = \frac{1}{\omega}\nabla\cdot \mathbf{u}_0, \]
Noting that $p(0,x) = p_0$ and $\mathbf{u}(0,x) = \mathbf{u}_0$, we find that $(p_0, \mathbf{u}_0)$ is a fixed point of $(p, \mathbf{u}) = \Pi (p,\mathbf{u})$ so we can simply apply the WaveHoltz iteration with real function $p$ and $\mathbf{u}$ using the time domain equations \eqref{eq:WAVE_EQuation_cons}, and extract the Helmholtz solution as
\[ \hat{p} = p - \frac{1}{i\omega}\nabla \cdot \mathbf{u}. \]

\bibliographystyle{siam}
\bibliography{references}
\end{document}